\documentclass[11pt]{article}
\usepackage{amsmath}
\usepackage{subfigure}
\usepackage{amsfonts}
\usepackage{float}
\usepackage{amsthm}
\usepackage{amssymb, setspace}
\usepackage{color,graphicx, epsfig, geometry, hyperref,fancyhdr}
\usepackage{mathrsfs}
\usepackage[T1]{fontenc}
\usepackage{latexsym,amssymb,amsmath,amsfonts,amsthm}
\usepackage{graphicx}
\usepackage{color}
\usepackage{soul}
\usepackage{caption}
\usepackage{dsfont}
\usepackage{multirow}
\usepackage{rotating}
\newtheorem{theorem}{Theorem}[section]

\newcommand{\dif}{\mathrm{d}}

\newcommand{\N}{\mathbb{N}}
\newcommand{\Z}{\mathbb{Z}}
\newcommand{\weakc}{\rightharpoonup}
\newcommand{\R}{\mathbb{R}}
\newcommand{\C}{\mathbb{C}}

\newcommand{\vecu}{{\bf u}}
\newcommand{\vecw}{  {\bf w} }
\newcommand{\vecv}{  {\bf v}  }

\newcommand{\E}{{\bf E}}
\newcommand{\F}{{\bf F}}
\newcommand{\grad}{\nabla}
\newcommand{\curl}{\nabla \times}
\newcommand{\ep}{\varepsilon}
\newcommand{\ov}{\overline}
\graphicspath{{pics/}}

\begin{document}
\setlength{\parskip}{1mm}
\setlength{\oddsidemargin}{0.1in}
\setlength{\evensidemargin}{0.1in}
\lhead{}
\rhead{}

{\bf \Large \noindent Analysis of two transmission eigenvalue problems with a coated boundary condition}

\begin{center}
Isaac Harris  \\
Department of Mathematics\\
Purdue University \\
West Lafayette, IN 47907\\
E-mail: harri814@purdue.edu
\end{center}

\begin{abstract}
\noindent In this paper, we investigate two transmission eigenvalue problems associated with the scattering of a media with a coated boundary. In recent years, there has been a lot of interest in studying these eigenvalue problems. It can be shown that the eigenvalues can be recovered from the scattering data and hold information about the material properties of the media one wishes to determine. Motivated by recent works we will study the electromagnetic transmission eigenvalue problem and scalar `zero-index' transmission eigenvalue problem for a media with a coated boundary. Existence of infinitely many real eigenvalues will be proven as well as showing that the eigenvalues depend monotonically on the refractive index and boundary parameter. Numerical examples in two spatial dimensions are presented for the scalar `zero-index' transmission eigenvalue problem. Also, in our investigation we prove that as the boundary parameter tends to zero and infinity we recover the classical eigenvalue problems. 
\end{abstract}


\section{Introduction}
In this paper, we study two interior transmission eigenvalue problems that arises from the scattering of a media with a coated boundary. Transmission eigenvalue problems have been a very active field of research in the theory of inverse scattering. See the manuscripts \cite{TE-book} for a detailed account of the main results and techniques for these eigenvalue problems for the scalar scattering problem. First, we will study the problem for the electromagnetic scattering that is analogous to the problem studied in \cite{te-cbc} and \cite{te-cbc2}. Next, being motivated by the new eigenvalue problems studied in \cite{Aniso-stekloff} and \cite{zi-te} we will consider the `zero-index' transmission eigenvalue problem for the scalar scattering problem. This transmission eigenvalue problem arises by artificially imbedding the scattering object mathematically in a background where in the scatterer the refractive index is {\it zero} (see for e.g. \cite{zi-te}). One can also artificially imbed  the scattering object in a background where the refractive index is negative (see for e.g. \cite{Aniso-stekloff}). The standard transmission eigenvalue problems are non self-adjoint as well as non-linear where as these modified transmission eigenvalue problems are linear. Since the transmission eigenvalue problem is non self-adjoint this can give rise to complex eigenvalues (see for e.g. \cite{te-cbc2} and \cite{complexte}). This makes these problems interesting to study analytically and numerically. See \cite{te-maxwell-fem} for a numerical method for computing the classical transmission eigenvalues for the Maxwell's system. These eigenvalue problems have become an important area of research in inverse scattering theory. In general, these eigenvalues can  determined from the scattering data and can be used to determine information about the underline scattering object (see for e.g. \cite{zi-te}, \cite{TE-book}, \cite{gpinvtev},  \cite{cavityhcs}, \cite{armin}, and \cite{te-maxwell-results}). In \cite{te-homog} the transmission eigenvalues are used to estimate the material properties of a highly oscillatory periodic scatterer. This suggests that the transmission eigenvalues can be used in many applications as a target signature to non-destructively test materials for defects.

We will first study the interior transmission eigenvalue problem associated with the following electromagnetic scattering problem. Let $D \subset \R^3$,  be a  simply connected domains with smooth $C^2$ boundary $\partial D$ where $\nu$ is the outward unit normal.  Now, let  $N$ denote the matrix valued refractive index and $\eta$ be a boundary coating parameter matrix.  The total field $\E(x,d,p) =\E^i(x,d,p) + \E^s(x,d,p)$ for $x \in \R^3$ where the incidence direction $d$ is such that $|d|=1$ and $p$ is the polarization with $p \cdot d =0$. The incident field is given by 
$${\color{black} \E^i(x,d,p) =\frac{\text{i}}{k} \curl \left( \curl p\mathrm{e}^{\mathrm{i} k {x} \cdot d} \right)}$$
with $k>0$ being the wave number. The total and scattered fields in $H_{loc}(\text{curl}, \R^3)$ satisfying
\begin{align}
\curl\curl \E - k^2 \E=0  \,\, \textrm{ in }  \R^3 \setminus \overline{D}  \quad  \text{and} \quad \curl\curl \E - k^2 N\, \E=0  \,\, &\textrm{ in } {D}  \label{direct1} \\
{\color{black} (\E_{+}-\E_{-}) \times \nu =0  \,\, \,  \text{and} \,\,\, \big( \curl  (\E_{+}-\E_{-}) \big) \times \nu = \eta \,   \big(  \left( \nu \times \E_{+} \right) \times \nu  \big) }\,\, &\textrm{ on } \partial D \label{direct3}\\
\lim\limits_{|x| \rightarrow \infty} |x| \big(  \left(\curl \E^s \right) \times \hat{x} -\mathrm{i}k \E^s \big)=0. \label{src}
\end{align}
The radiation condition \eqref{src} is satisfied uniformly with respect to the direction $\hat{x}=x/|x|$. See \cite{LSM-maxwell-book} for the direct and inverse media scattering problem for determining $D$ using the Linear Sampling Method. For the case where {\color{black}$\eta$ is purely imaginary} the above problem represents scattering by an anisotropic media with a coated boundary given by the impedance boundary condition. Transmission eigenvalue problems are studied in the context of inverse scattering since they {\color{black} are related} to the wave numbers where the associated far-field operator fails to be injective with a dense range. Due to the loss of injective and density of the range one has that some qualitative reconstruction methods like the Linear Sampling Method fail at the corresponding wavenumber.  

The second transmission eigenvalue problem we consider is associated with the scalar scattering problem with a conductive boundary. We are interested mainly in the inverse scattering problem of determining information about the coefficients from the far-field data. Therefore, we assume that the support of the scatterer $D$ is known and we will investigate what information about the material properties can be obtained from the zero-index transmission eigenvalues. See \cite{fm-cbc} for the reconstruction of the scatterer from the  far-field data via the Factorization Method.  The  eigenvalue problem we consider here corresponds to the direct scattering problem: find the total field $u \in H^1_{loc}(\R^m)$ for $m=2,3$ such that  
\begin{align}
\Delta u +k^2 u=0  \quad \textrm{ in }  \R^m \setminus \overline{D}  \quad  \text{and} \quad  \Delta u +k^2 nu=0  \quad &\textrm{ in } \,  {D}  \label{direct4}\\
 u_+-u_-=0  \quad  \text{and} \quad {\partial_\nu u_+} + \eta u_+ = {\partial_\nu u_-} \quad &\textrm{ on } \partial D \label{direct5}
\end{align}
with $u=u^s+u^i$. The incident field is given by $u^i=\text{e}^{\text{i} k x \cdot d}$ with the incident direction $d$ given by a point on the unit circle/sphere. The scattered field $u^s$ satisfies the Sommerfeld radiation condition given by
$$\lim\limits_{|x| \rightarrow \infty} |x|^{(m-1)/2} \left( \frac{\partial u^s}{\partial |x|} -\text{i}k u^s \right)=0$$
which is satisfied uniformly with respect to the direction $\hat{x}=x/|x|$.
Let $D \subset \R^m$ be a bounded simply connected open set with $\nu$ the unit outward normal to $\partial D \in C^2$. Here, we let $n$ denote the refractive index and $\eta$ denotes the conductivity parameter on $\partial D$. The problem under consideration is to study the eigenvalue problem that one gets from artificially imbedding the scattering object $D$  in a background material where the refractive index is 1 on the exterior of $D$ and is 0 on the interior. This problem was studied in \cite{zi-te} with $\eta =0$ where the authors recovered the plate buckling eigenvalue problem. In our investigation, we will study the so-called zero-index transmission eigenvalue problem with both $n$ and $\eta$.

\section{Electromagnetic transmission eigenvalue problem}\label{sect-definition}

In this section, we will rigorously define the Electromagnetic transmission eigenvalue problem as well as develop the variational formulation in the appropriate function spaces. 
Here we define the Hilbert spaces ${\bf L}^2(D)=\big[L^2(D)\big]^3$ and ${\bf H}^1(D)=\big[H^1(D)\big]^3$ with the standard inner-products. We now define the Hilbert spaces $H(\text{curl} ,D) = \left\{ \vecu \in {\bf L}^2(D) \, : \, \curl \vecu \in {\bf L}^2(D)\right\} $
and $H_0(\text{curl} ,D) = \left\{ \vecu \in H(\text{curl} ,D)  \, : \,  \vecu \times \nu = 0 \, \text{ on } \, \partial D  \right\}$
with the inner-product 
$$(\vecu , \boldsymbol{\varphi})_{H(\text{curl} ,D)} =(\vecu , \boldsymbol{\varphi})_{{\bf L}^2(D)} +(\curl \vecu , \curl \boldsymbol{\varphi})_{{\bf L}^2(D)}$$
as well as $H(\text{curl$^2$} ,D) =\left\{ \vecu \in H(\text{curl} ,D)  \, : \,  \curl \vecu \in H(\text{curl} ,D)  \right\}$
equipped with the inner-product 
$$(\vecu , \boldsymbol{\varphi})_{H(\text{curl$^2$} ,D)} = (\vecu , \boldsymbol{\varphi})_{{\bf L}^2(D)} +(\curl \vecu , \curl \boldsymbol{\varphi})_{H(\text{curl} ,D)}.$$

The interior transmission eigenvalue problem under consideration is to determine the values of $k \in \C$ such that there exists a nontrivial solution to
\begin{align}
\curl\curl \vecw - k^2 N\, \vecw=0 \quad \text{and} \quad  \curl\curl\vecv - k^2 \vecv=0  \quad &\textrm{ in } \,  D \label{teprob1} \\
{\color{black} (\vecw-\vecv) \times \nu =0  \quad  \text{and} \quad  \big( \curl  (\vecw-\vecv) \big) \times \nu = \eta\, \big(  \left( \nu \times \vecw \right) \times \nu \big) }  \quad &\textrm{ on } \partial D.  \label{teprob2} 
\end{align} 
We say that $k$ is an {\it interior transmission eigenvalue} if there is a nontrivial pair of functions $(\vecw,\vecv) \in {\bf L}^2(D) \times {\bf L}^2(D)$ such that the difference $\vecw-\vecv \in X(D)$ where we define the Hilbert space 
$$X(D)=H(\text{curl$^2$} ,D) \cap H_0(\text{curl} ,D)$$ 
equipped with the $H(\text{curl$^2$} ,D)$ inner product. We assume $N(x) \in L^{\infty} (D,\R^{3 \times 3})$ and $\eta(x) \in L^{\infty} (\partial D,\R^{3 \times 3})$ are symmetric uniformly positive definite matrices. For the case where {\color{black} $\eta$ is purely imaginary} the above problem can be derived from the electromagnetic scattering by an inhomogeneous anisotropic medium coated with a highly conductive layer(see for e.g. \cite{CH-directproblem}). Here we assume that there are positive constants $n_{\text{min}}$ and $n_{\text{max}}$ such that 
$$ n_{\text{min}} |\xi|^2 \leq N(x)\xi \cdot \ov{\xi} \leq n_{\text{max}}  |\xi|^2  \quad \text{ for all } \, \xi \in \C^3 \quad \text{ a.e.} \, \, x \in \overline{D}.$$
Similarly for $\eta$ we assume that there are positive constants $\eta_{\text{min}}$ and $\eta_{\text{max}}$ such that 
$$ \eta_{\text{min}} |\xi|^2 \leq \eta(x)\xi \cdot \ov{\xi} \leq \eta_{\text{max}}  |\xi|^2  \quad \text{ for all } \, \xi \in \C^3 \quad \text{ a.e.} \, \, x \in \partial D.$$

In order to study this transmission eigenvalue problem we will follow the analytic framework in \cite{chtevexist} by considering the equivalent `quad-curl' formulation of the problem \eqref{teprob1}--\eqref{teprob2}. In \cite{chtevexist} the transmission eigenvalue problem for $\eta =0$ was studied and the corresponding interior transmission problem was studied in \cite{maxwellitp} where here we modify the analysis {\color{black}for our case.}

The eigenfunctions $\vecw$ and $\vecv$ solve \eqref{teprob1} in the distributional sense and now let $\vecu = \vecw-\vecv$ which gives that $\vecu \in X(D)$ satisfies 
$$ \curl \curl \vecu -k^2 \vecu = k^2 (N-I) \vecw \quad \text{ in } \,\, D$$ 
Where $I$ denotes the identity matrix. Provided that either $n_{\text{max}}<1$ or $n_{\text{min}}>1$ we have that $N-I$ is an invertable matrix a.e. in $\overline{D}$ which implies that
\begin{align}
\big(\curl \curl -k^2 N\big) (N-I)^{-1} \big( \curl \curl \vecu -k^2 \vecu \big) = 0 \quad \text{ in } \,\,  D \label{teprob3}
\end{align}
and the boundary condition \eqref{teprob2} becomes 
\begin{align}
{\color{black} k^2 ( \curl \vecu ) \times \nu = \eta\, \Big(\nu \times (N-I)^{-1} \big( \curl \curl \vecu -k^2 \vecu \big)  \Big)\times \nu } \quad \text{ on } \,\,  \partial D \label{teprob4}
\end{align}
where the equality in \eqref{teprob4} is understood in the trace sense. The eigenfunctions $\vecw$ and $\vecv$ can be determined from $\vecu$ through 
$$ k^2 \vecw =(N-I)^{-1} \big( \curl \curl \vecu -k^2 \vecu \big) \quad \text{and} \quad k^2 \vecv =(N-I)^{-1} \big( \curl \curl \vecu -k^2 N\vecu \big).$$ 
This implies that there exists nontrivial solutions to \eqref{teprob1}--\eqref{teprob2} if and only if there is a nontrivial solution to \eqref{teprob3}--\eqref{teprob4} giving the equivalent of the two eigenvalue problems. Therefore, we will analyze the variational formulation of \eqref{teprob3}--\eqref{teprob4} in the Hilbert space $X(D)$.   To do so, we need the following results for our variational space. 

\begin{theorem}\label{varspace}
The norms $\| \vecu \|^2_{H(\text{curl$\,^2$} ,D)}$ and $\|\vecu\|^2_{{\bf L}^2(D)} + \| \curl \curl \vecu \|^2_{{\bf L}^2(D)}$ are equivalent in $X(D)$. Also, For all $ \vecu \in X(D)$ we have that $\curl \vecu \in {\bf H}^1(D)$ and for all $\vecu \in X(D)$ satisfy the inequality $\| \curl \vecu\|^2_{{\bf H}^1(D)} \leq C \| \curl \curl \vecu \|^2_{{\bf L}^2(D)}.$
\end{theorem}
\begin{proof}
The equivalence of the norms is a simple consequence of Green's Theorem and Young's Inequality. Now notice that $\grad \cdot ( \curl \vecu) = 0$ in $D$ and $(\curl \vecu) \cdot \nu = \grad_{\partial D} \cdot (\vecu \times \nu)=0$ on $\partial D$ in the weak sense where $\grad_{\partial D}$ is the surface gradient. This gives that $\curl \vecu \in H(\text{curl},D) \cap H_0(\text{div} ,D)$ which is continuously imbedded in  ${\bf H}^1(D)$ by \cite{vector-embedding}. Therefore, by applying the Friedrich's inequality given in Corollary 1 of \cite{maxwell-ineq} we have that  $\| \curl \vecu\|^2_{{\bf H}^1(D)} \leq C \| \curl \curl \vecu \|^2_{{\bf L}^2(D)}$
since $D$ is simple connected with $C^2$ boundary, proving the claim. 
\end{proof}

We now derive that variational formulation associated with \eqref{teprob3}--\eqref{teprob4}. In order to obtain the variational formulation multiple elementary vector identities are used involving the dot and cross product. Let $\boldsymbol{\varphi} \in X(D)$ then taking the dot product with \eqref{teprob3} and $\ov{\boldsymbol{\varphi}}$ then integrating over $D$ using Green's Theorem gives that
\begin{align*}
\int\limits_D \ov{\boldsymbol{\varphi}} \cdot (\curl \curl \F &- N k^2 \F) \, \dif x = \int\limits_D \F \cdot (\curl \curl \ov{\boldsymbol{\varphi}}  - N k^2 \ov{\boldsymbol{\varphi}} ) \, \dif x\\
                 				 		&+ \int\limits_{\partial D}  \ov{\boldsymbol{\varphi}} \cdot \nu \times ( \curl \F )  \, \dif s+\int\limits_{\partial D} \curl \ov{\boldsymbol{\varphi}} \cdot (  \nu \times \F )  \, \dif s
\end{align*}
where have let $\F =(N-I)^{-1} \big( \curl \curl \vecu -k^2 \vecu \big)$. For the fist boundary integral we have that 
$$\int\limits_{\partial D}  \ov{\boldsymbol{\varphi}} \cdot \big( \nu \times ( \curl \F ) \big) \, \dif s =\int\limits_{\partial D}  (\ov{\boldsymbol{\varphi}} \times \nu) \cdot  (\curl \F )  \, \dif s = 0$$
since $\boldsymbol{\varphi} \in X(D)$. Similarly for the second boundary integral we have that 
$$ {\color{black} \int\limits_{\partial D} (\curl \ov{\boldsymbol{\varphi}} ) \cdot (  \nu \times \F )  \, \dif s = \int\limits_{\partial D} \F \cdot \big( (\curl \ov{\boldsymbol{\varphi}}) \times \nu \big)   \, \dif s = \int\limits_{\partial D} \big( (\nu \times \F) \times \nu  \big) \cdot  \big((\curl \ov{\boldsymbol{\varphi}}) \times \nu \big)  \, \dif s }$$
where we have used that $\F = (\F \cdot \nu) \nu +(\nu \times \F) \times \nu$ as well as the fact that $(\curl \ov{\boldsymbol{\varphi}}) \times \nu$ and $(\F \cdot \nu) \nu$ are perpendicular with respect to the dot product. 
The variational form is therefore given by
\begin{align} \label{varform}
\int\limits_D (N-I)^{-1}(\curl \curl \vecu &- k^2 \vecu) \cdot ({\curl \curl \ov{\boldsymbol{\varphi}} } - k^2 N \ov{\boldsymbol{\varphi}}) \, \dif x \nonumber \\
&+  k^2 \int\limits_{\partial D}{\eta^{-1}} {\color{black} \big( ( \curl \vecu ) \times \nu \big)\cdot  \big( (\curl \ov{\boldsymbol{\varphi} }) \times \nu \big) } \, \dif s= 0  
\end{align}
for all $\boldsymbol{\varphi} \in X(D)$ where the boundary integral incorporates the boundary condition \eqref{teprob4}. Note that the boundary integral in \eqref{varform} is well defined since both $\curl \boldsymbol{\varphi}$ and $\curl \vecu$ are in  ${\bf H}^1(D)$ which implies that the trace of their components is in $H^{1/2}(\partial D) \subset L^2(\partial D)$. We have also used the fact that we have assumed $\eta$ is uniformly positive definite in $\partial D$ giving that $\eta^{-1} \in L^{\infty} (\partial D,\R^{3 \times 3})$ such that  
$$\frac{  |\xi|^2}{\eta_{\text{max}}} \leq \eta^{-1}(x) \xi \cdot \ov{\xi} \leq \frac{  |\xi|^2}{\eta_{\text{min}}}  \quad \text{ for all } \, \xi \in \C^3 \quad \text{ a.e.} \, \,  x \in \partial D.$$
Also, notice that since we assume either $n_{\text{max}}<1$ or $n_{\text{min}}>1$ we have that the matrix $(N-I)^{-1} \in  L^{\infty} (D,\R^{3 \times 3})$ such that 
$$ \frac{  |\xi|^2}{n_{\text{max}}-1} \leq \big(N(x)-I\big)^{-1} \xi \cdot \ov{\xi} \leq \frac{  |\xi|^2}{n_{\text{min}}-1}  \,\, \text{ for } \,\, n_{\text{min}}>1$$
or 
$$  \frac{  |\xi|^2}{1-n_{\text{min}} } \leq \big(I-N(x)\big)^{-1} \xi \cdot \ov{\xi} \leq \frac{  |\xi|^2}{1-n_{\text{max}} }\,\,  \text{ for } \,\, n_{\text{max}}<1$$
 for all $\xi \in \C^3$ and a.e. $x \in \overline{D}$. Therefore, the volume integral is well defined in the variational space. 

\subsection{Existence of electromagnetic transmission eigenvalues}\label{sect-existence}
In this section, we will prove the existence of infinitely many real transmission eigenvalues. To do so, we will apply the theory used for studying the case where $\eta=0$ where one only has the volume term in \eqref{varform}. In order to use the results in \cite{chtevexist} we need to split the variational formulation into coercive and compact parts. Throughout this section we assume that the coefficient matrices $N$ and $\eta$ are real and symmetric with either $n_{\text{max}}<1$ or $n_{\text{min}}>1$ and $\eta_{\text{min}} >0$. By manipulating the variational for \eqref{varform} one can show that $k$ is an interior transmission eigenvalue with the corresponding eigenfunction $\vecu \in X(D)$ provided that 
\begin{equation}
\mathcal{A}_k(\vecu,\boldsymbol{\varphi}) -k^2\mathcal{B}(\vecu,\boldsymbol{\varphi}) =0 \quad \text{ for all } \quad  \boldsymbol{\varphi} \in X(D), \quad \text{when} \, \, \, n_{\text{min}}>1,
\end{equation}
or 
\begin{equation}
\widetilde{ \mathcal{A}}_k(\vecu,\boldsymbol{\varphi}) -k^2 \widetilde{ \mathcal{B}} (\vecu,\boldsymbol{\varphi}) =0 \quad \text{ for all } \quad  \boldsymbol{\varphi} \in X(D), \quad \text{when} \, \, \, n_{\text{max}}<1.
\end{equation}
The sesquilinear forms on $X(D) \times X(D) \longmapsto \C$ are derived from manipulating the variational formulation \eqref{varform} and are given by
\begin{align}
 \mathcal{A}_k(\vecu,\boldsymbol{\varphi}) &=  \int\limits_D (N-I)^{-1} \big( \curl \curl \vecu -k^2 \vecu \big) \cdot  \big( \curl \curl  \overline{\boldsymbol{\varphi}} -k^2  \overline{\boldsymbol{\varphi}} \big)  \, \dif x \nonumber \\ 
& +k^4 \int\limits_D  \vecu \cdot  \overline{\boldsymbol{\varphi}} \, \dif x + k^2\int\limits_{\partial D} \eta^{-1} {\color{black} \big( ( \curl \vecu ) \times \nu \big)\cdot  \big( (\curl \ov{\boldsymbol{\varphi} }) \times \nu \big) ) \, \dif s}, \label{A1} 
\end{align}
\begin{align}
\widetilde{\mathcal{A}}_k (\vecu,\boldsymbol{\varphi})&= \int\limits_D N(I-N)^{-1} \big( \curl \curl \vecu -k^2 \vecu \big) \cdot  \big( \curl \curl  \overline{\boldsymbol{\varphi}} -k^2  \overline{\boldsymbol{\varphi}} \big)  \, \dif x \nonumber \\ 
&\hspace{1in} + \int\limits_D {\color{black} \big( \curl ( \curl \vecu ) \big)\cdot  \big(\curl (\curl \overline{\boldsymbol{\varphi}}) \big)} \, \dif x, \label{A2} \\
\mathcal{B}(\vecu,\boldsymbol{\varphi})&=\int\limits_D  {\color{black} (\curl \vecu ) \cdot (\curl  \overline{\boldsymbol{\varphi}} ) } \, \dif x,  \label{B1} 
\end{align}
and 
\begin{align}
 \widetilde{\mathcal{B}}(\vecu,\boldsymbol{\varphi}) &= \int\limits_D \curl \vecu \cdot \curl  \overline{\boldsymbol{\varphi}}  \, \dif x + \int\limits_{\partial D} \eta^{-1} {\color{black} \big( ( \curl \vecu ) \times \nu \big)\cdot  \big( (\curl \ov{\boldsymbol{\varphi} }) \times \nu \big)}  \, \dif s.  \label{B2}
\end{align}

Since we have that $\eta^{-1} \in L^{\infty} (\partial D,\R^{3 \times 3})$ and $(N-I)^{-1} \in  L^{\infty} (D,\R^{3 \times 3})$ we that the sesquilinear forms are bounded. By employing the Riesz representation theorem we can define the bounded linear operators $\mathbb{A}_k$, $\widetilde{\mathbb{A}}_k$, $\mathbb{B}$, and $\widetilde{ \mathbb{B}}:X(D) \longmapsto X(D)$ that representation the sesquilinear forms where
\begin{align*}
\left( {\mathbb{A}}_k \vecu,\boldsymbol{\varphi} \right)_{X(D)}={ \mathcal{A}}_k(\vecu,\boldsymbol{\varphi}),  \quad \big( \widetilde{\mathbb{A}}_k \vecu,\boldsymbol{\varphi} \big)_{XD)}=\widetilde{ \mathcal{A}}_k(u,\varphi), 
\end{align*}
\begin{align*}
\left( {\mathbb{B}} \vecu,\boldsymbol{\varphi}\right)_{X(D)}={ \mathcal{B}}(\vecu,\boldsymbol{\varphi})  \quad \text{ and } \quad \big( \widetilde{\mathbb{B}}\vecu,\boldsymbol{\varphi} \big)_{X(D)}=\widetilde{ \mathcal{B}}(\vecu,\boldsymbol{\varphi})
\end{align*}
for all $ \vecu $ and $\boldsymbol{\varphi} \in X(D)$. It is clear from the definition for operators that the mappings $k \longmapsto \mathbb{A}_k$ and  $k \longmapsto \widetilde{\mathbb{A}}_k$ from the complex plane to the set of bounded linear operators on $X(D)$ is analytic.

In order to prove the existence of real transmission eigenvalues we will use Theorem 2.3 of \cite{chtevexist}. Therefore, we first show that $\mathbb{B}$ and $\widetilde{ \mathbb{B}}$ are compact. By the compact imbedding of ${\bf H}^1(D)$ into ${\bf L}^2(D)$ and Theorem \ref{varspace} we clearly have the compactness of operator $\mathbb{B}$. Also by Theorem \ref{varspace} we have that for all $\vecu \in X(D)$ that the  trace of $\curl \vecu$ on the boundary $\partial D$ has components in $H^{1/2}(\partial D)$. Using the inequality 
$$\int\limits_{\partial D} \eta^{-1} \big|( \curl \vecu ) \times \nu \big|^2 \, \dif s \leq \eta^{-1}_{\text{min}}\int\limits_{\partial D}  \big|\curl \vecu \big|^2 \, \dif s$$ 
and the compact imbedding of  $H^{1/2}(\partial D)$ into $L^2(\partial D)$ gives that both the volume and boundary terms in $\widetilde{ \mathbb{B}}$ can be represented by compact operators. Since $\eta$ is a real symmetric matrix it is clear that both $\mathbb{B}$ and $\widetilde{ \mathbb{B}}$ are self-adjoint operators since the sesquilinear forms ${ \mathcal{B}}( \cdot \, ,\cdot)$ and $\widetilde{ \mathcal{B}}( \cdot \, ,\cdot)$ are Hermitian. It is also clear by the definition that $\mathbb{B}$ and $\widetilde{ \mathbb{B}}$ are non-negative operators since $\eta$ is a uniformly positive definite matrix. This gives the following result. 
\begin{theorem}\label{Bforms}
{\color{black} Assume that  $\eta(x) \in L^{\infty} (\partial D,\R^{3 \times 3})$ is uniformly positive definite. Then the operators $\mathbb{B}$ and $\widetilde{ \mathbb{B}} :X(D) \longmapsto X(D)$ are self-adjoint, non-negative and compact. }
\end{theorem}

We now turn our attention to studying the operators $\mathbb{A}_k$ and $\widetilde{\mathbb{A}}_k$. In order to apply Theorem 2.3 of \cite{chtevexist} to our transmission eigenvalue problem we need to show that the operators $\mathbb{A}_k$ and $\widetilde{\mathbb{A}}_k$ are self-adjoint and coercive for all positive values of $k$. To do so, we first notice that by Theorem \ref{varspace} we can take 
$$\|\vecu\|^2_{X(D)}=\|\vecu\|^2_{{\bf L}^2(D)} + \| \curl \curl \vecu \|^2_{{\bf L}^2(D)}.$$
With this we are now ready to study the operators $\mathbb{A}_k$ and $\widetilde{\mathbb{A}}_k$. 
\begin{theorem}\label{Aforms}
{\color{black} Assume that $N(x) \in L^{\infty} (D,\R^{3 \times 3})$ and $\eta(x) \in L^{\infty} (\partial D,\R^{3 \times 3})$ are symmetric uniformly positive definite.  Then the operators $\mathbb{A}_k$ and $\widetilde{ \mathbb{A}}_k :X(D) \longmapsto X(D)$ are self-adjoint for all $k \in \R$. Moreover, assume that either $n_{\text{min}}>1$ or $n_{\text{max}}<1$. Then the operators are coercive for all $k \in \R \setminus \{0\}$ satisfying the estimates }
$$\left( {\mathbb{A}}_k \vecu,\vecu \right)_{X(D)} \geq \frac{1}{n_{\text{max}} +1}  \| \curl \curl \vecu \|^2_{{\bf L}^2(D)} + \frac{k^4}{2} \|\vecu\|^2_{{\bf L}^2(D)} \quad \text{ for all } \quad  \vecu \in X(D)$$
and 
$$\big( \widetilde{\mathbb{A}}_k \vecu,\vecu \big)_{XD)} \geq \frac{1}{2}  \| \curl \curl \vecu \|^2_{{\bf L}^2(D)} + {k^4} \frac{n_{\text{min}} }{n_{\text{min}} +1} \|\vecu\|^2_{{\bf L}^2(D)} \quad \text{ for all } \quad  \vecu \in X(D)$$
\end{theorem}
\begin{proof}
To begin, we notice that the sesquilinear forms ${ \mathcal{A}}_k( \cdot \, ,\cdot)$ and $\widetilde{ \mathcal{A}}_k( \cdot \, ,\cdot)$ are Hermitian for all $k \in \R$ which implies that $\mathbb{A}_k$ and $\widetilde{\mathbb{A}}_k$ are self-adjoint. Now we prove the coercivity estimates. Therefore, we begin with $\mathbb{A}_k$ and by definition we have that 
 \begin{align*}
\left( {\mathbb{A}}_k \vecu, \vecu \right)_{X(D)} &\geq \int\limits_D (N-I)^{-1} \big| \curl \curl \vecu -k^2 \vecu \big|^2  \, \dif x + k^4 \int\limits_D  |\vecu|^2\, \dif x\\
&\hspace{-0.5in} \geq \alpha \int\limits_D \big| \curl \curl \vecu -k^2 \vecu \big|^2  \, \dif x + k^4 \int\limits_D  |\vecu|^2\, \dif x \, \text{ where } \, \alpha = \frac{  1 }{n_{\text{max}} - 1}. 
\end{align*}
Following in the same way as in \cite{CH-itp} {\color{black}by appealing to the Cauchy-Schwarz inequality and Young's inequality we have that
\begin{align*}
\left( {\mathbb{A}}_k \vecu, \vecu \right)_{X(D)} & \geq \alpha \| \curl \curl \vecu \|^2_{{\bf L}^2(D)} -2\alpha k^2 \| \curl \curl \vecu \|_{{\bf L}^2(D)} \|  \vecu \|_{{\bf L}^2(D)}+ (\alpha+1)k^4 \|  \vecu \|^2_{{\bf L}^2(D)}\\
&\geq \left(\alpha -\frac{\alpha^2}{\ep} \right) \| \curl \curl \vecu \|^2_{{\bf L}^2(D)} + (\alpha +1 - \ep) k^4\|  \vecu \|^2_{{\bf L}^2(D)}  
\end{align*}
provided that} $\alpha < \ep <\alpha + 1$. We choose $\ep = \alpha + 1/2$ which gives the coercivity estimate for $ {\mathbb{A}}_k$. Now we consider $\widetilde{ \mathbb{A}}_k$ and similarly by definition we have that 
\begin{align*}
\big(\widetilde{\mathbb{A}}_k \vecu, \vecu \big)_{X(D)} &= \int\limits_D N(I-N)^{-1} \big| \curl \curl \vecu -k^2 \vecu \big|^2 \, \dif x + \int\limits_D  \left|\curl \curl \vecu \right|^2 \, \dif x\\
&\hspace{-0.5in} \geq \beta \int\limits_D \big| \curl \curl \vecu -k^2 \vecu \big|^2 \, \dif x + \int\limits_D  \left|\curl \curl \vecu \right|^2 \, \dif x \, \text{ where } \, \beta = \frac{ n_{\text{min}}  }{1-n_{\text{min}} }.
\end{align*}
Similarly as above 
\begin{align*}
\big(\widetilde{\mathbb{A}}_k \vecu, \vecu \big)_{X(D)} &\geq  (\beta +1 - \ep) \| \curl \curl \vecu \|^2_{{\bf L}^2(D)} +  \left(\beta -\frac{\beta^2}{\ep} \right) k^4\|  \vecu \|^2_{{\bf L}^2(D)}  
\end{align*}
for any $\beta < \ep <\beta + 1$. We choose $\ep = \beta + 1/2$ which proves the claim. 
\end{proof}

Theorems \ref{Bforms} and \ref{Aforms} gives that the operators $\mathbb{B}$, $\widetilde{ \mathbb{B}}$, $\mathbb{A}_k$ and  $\widetilde{\mathbb{A}}_k$ satisfy the assumptions needed to apply Theorem 2.3 of \cite{chtevexist} to prove the existence of transmission eigenvalues. In order to apply this result we need to show that the operators $\mathbb{A}_k -k^2 \mathbb{B}$ and $\widetilde{\mathbb{A}}_k -k^2 \widetilde{\mathbb{B}}$ are positive on $X(D)$ for some $k_1$ and are non-positive on a subspace of $X(D)$ for some $k_2$. We note that the real transmission eigenvalues are solutions to the equation $\lambda_j (k)-k^2=0$ where $\lambda_j$ are the generalized eigenvalues such that there exists a nontrivial $\vecu \in X(D)$ where 
\begin{align}
\mathbb{A}_k \vecu= \lambda_j \mathbb{B} \vecu \,\, \text{ for } \, \, 1<n_{\text{min}} \quad \text{ or } \quad \widetilde{\mathbb{A}}_k \vecu= \lambda_j \widetilde{\mathbb{B}} \vecu  \,\, \text{ for } \, \, n_{\text{max}}<1. \label{geneig} 
\end{align}
Since the ${ \mathcal{A}}_k( \cdot \, ,\cdot)$ and $\widetilde{ \mathcal{A}}_k( \cdot \, ,\cdot)$ depend continuously on $k$ the existence of real transmission eigenvalues comes from appealing to the Intermediate Value Theorem. We now show that both $\mathbb{A}_k -k^2 \mathbb{B}$ and $\widetilde{\mathbb{A}}_k -k^2 \widetilde{\mathbb{B}}$ are positive operators for sufficiently small values of $k>0$. 
\begin{theorem}\label{positive}
{\color{black} Assume that $N(x) \in L^{\infty} (D,\R^{3 \times 3})$ and $\eta(x) \in L^{\infty} (\partial D,\R^{3 \times 3})$ are symmetric uniformly positive definite and that either $n_{\text{min}}>1$ or $n_{\text{max}}<1$. Then for all $k>0$ sufficiently small there exists $\delta >0$ such that for all $\vecu \in X(D)$ }
$$\mathcal{A}_k(\vecu,\vecu) -k^2\mathcal{B}(\vecu,\vecu) \geq \delta \| \vecu \|^2_{X(D)}   \quad \text{ or  } \quad \widetilde{ \mathcal{A}}_k(\vecu,\vecu) -k^2 \widetilde{ \mathcal{B}} (\vecu,\vecu)  \geq \delta \| \vecu \|^2_{X(D)}.$$
\end{theorem}
\begin{proof}
To begin, we start with the simpler case when $n_{\text{min}}>1$ so we use Theorem \ref{Aforms} to estimate 
\begin{align*}
\mathcal{A}_k(\vecu,\vecu) -k^2\mathcal{B}(\vecu,\vecu) &\geq \frac{1}{n_{\text{max}} +1}   \| \curl \curl \vecu \|^2_{{\bf L}^2(D)} + \frac{k^4}{2} \|\vecu\|^2_{{\bf L}^2(D)} - k^2 \| \curl \vecu\|^2_{{\bf L}^2(D)}\\
										      &\geq \left( \frac{1}{n_{\text{max}} +1}   -Ck^2 \right) \| \curl \curl \vecu \|^2_{{\bf L}^2(D)} + \frac{k^4}{2} \|\vecu\|^2_{{\bf L}^2(D)} 
\end{align*}
where $C>0$ is the constant from Friedrich's inequality in Theorem \ref{varspace} which gives result for this case. Now consider the case where $n_{\text{max}}<1$ where we again use Theorem \ref{Aforms} and Friedrich's inequality to obtain the estimate
\begin{align*}
\widetilde{ \mathcal{A}}_k(\vecu,\vecu) -k^2 \widetilde{ \mathcal{B}} (\vecu,\vecu) &\\
&\hspace{-1in}\geq \left( \frac{1}{2}   -Ck^2 \right) \| \curl \curl \vecu \|^2_{{\bf L}^2(D)} +{k^4} \frac{n_{\text{min}} }{n_{\text{min}} +1}\|\vecu\|^2_{{\bf L}^2(D)} -\frac{k^2}{\eta_{\text{min}} } \|  \curl \vecu \|^2_{{\bf L}^2(\partial D)} \\
 &\hspace{-1in}\geq \left[\frac{1}{2}  -C k^2 \left(1+\frac{c}{\eta_{\text{min}}} \right) \right] \| \curl \curl \vecu \|^2_{{\bf L}^2(D)} + {k^4} \frac{n_{\text{min}} }{n_{\text{min}} +1} \|\vecu\|^2_{{\bf L}^2(D)}
\end{align*}
where $c>0$ is the constant for the Trace Theorem such that $\|\boldsymbol{\varphi} \|^2_{{\bf L}^2(\partial D)} \leq c \|\boldsymbol{\varphi} \|^2_{{\bf H}^1(D)}$ for all $\boldsymbol{\varphi} \in {\bf H}^1(D)$  and again $C>0$ is the constant from Friedrich's inequality, proving the claim. 
\end{proof}

We are now ready to prove the main result of this section i.e. there exists infinitely many real transmission eigenvalues. To do so, we must show that there is an infinite dimensional subset of $X(D)$ and some value $k>0$ for which that operators $\mathbb{A}_k -k^2 \mathbb{B}$ and $\widetilde{\mathbb{A}}_k -k^2 \widetilde{\mathbb{B}}$ are non-positive then Theorem 2.3 of \cite{chtevexist} gives the result. 
\begin{theorem}  \label{exists}
{\color{black} Assume that $N(x) \in L^{\infty} (D,\R^{3 \times 3})$ and $\eta(x) \in L^{\infty} (\partial D,\R^{3 \times 3})$ are symmetric uniformly positive definite and that either $n_{\text{min}}>1$ or $n_{\text{max}}<1$. Then there exists infinitely many real transmission eigenvalues. }
\end{theorem}
\begin{proof}
We only present the proof for the case when $n_{\text{min}}>1$ and the other case follows from similar arguments. We let $B_j=B(x_j , \ep):=\{ x \in \R^3 : |x-x_j | <  \ep \}$ where $x_j \in D$ and $\ep>0$. Define $M(\ep)$ the supremum of the number of disjoint balls $B_j$, i.e., $\ov{B_i} \cap \overline{B_j} = \emptyset$, such that $\overline{B_j} \subset D$. Using  separation of variables we have that there exists transmission eigenvalues to 
\begin{align}
{\color{black} \curl \curl \vecw_j -k^2 n_{\text{min}} \vecw_j=0 \quad \text{and} \quad \curl \curl  \vecv_j - k^2 \vecv_j=0}  \quad &\textrm{ in } \,  B_j, \label{teball1} \\
( \vecw_j-\vecv_j)\times \nu =0  \quad  \text{and} \quad  \curl  (\vecw_j-\vecv_j) \times \nu = 0 \quad &\textrm{ on } \partial B_j.  \label{teball2} 
\end{align}
We define $\vecu_j$ as the difference $\vecu_j=\vecw_j-\vecv_j$ in $B_j$ and $\vecu_j = 0$ in $D \setminus {B_j}$ for any transmission eigenvalue of \eqref{teball1}--\eqref{teball2}. It is clear that 
$$\vecu_j \in H_0(\text{curl$^2$} ,D)=\left\{ \vecu \in H_0(\text{curl} ,D)  \, : \,  \curl \vecu \in H_0(\text{curl} ,D)  \right\} \subset X(D).$$
This implies that $X_{M(\ep)} = \text{ span}\{ {\vecu}_1 , {\vecu}_2 , \cdots,  {\vecu}_{M(\ep)}  \}$ forms an $M(\ep)$ dimensional subspace of $X(D)$ since the support of the basis functions are disjoint gives that they are orthogonal. Simple calculations as in Section \ref{sect-definition} gives that for every eigenvalue of \eqref{teball1}--\eqref{teball2}
$${\color{black} 0 = \int\limits_{B_j}  \frac{1}{ n_{\text{min}}-1} | \curl \curl {\vecu}_j - k^2 {\vecu}_j |^2 +k^4 |{\vecu}_j|^2  - k^2 |\curl {\vecu}_j|^2 \,  \dif x}.$$
Denoting $k_\ep$ as the first transmission eigenvalue of \eqref{teball1}--\eqref{teball2} for the ball $B_j$ with the corresponding eigenfunctions $\vecu_j$. Now using the fact that $\vecu_j$ are supported in $B_j$ along with 
$$\big(N(x)-I\big)^{-1} \xi \cdot \ov{\xi} \leq \frac{  |\xi|^2}{n_{\text{min}}-1}  \quad \text{ for all  } \,\, \xi \in \C^3 \quad \text{ and a.e. } \,\, x \in D$$
we have that 
\begin{align*}
\mathcal{A}_{k_{\ep}} ({\vecu}_j, {\vecu}_j ) -k_{\ep}^2\mathcal{B}({\vecu}_j,{\vecu}_j) & \\
&\hspace{-1.2in} =\int\limits_D (N-I)^{-1} \big| \curl \curl \vecu_j -k_{\ep}^2 \vecu_j \big|^2 + k_{\ep}^4  |\vecu_j |^2  - k_{\ep}^2  \big| \curl \vecu_j \big|^2\, \dif x \\
 &\hspace{-1.2in}\leq {\color{black} \int\limits_{B_j} \frac{1}{n_{\text{min}}-1} | \curl \curl  {\vecu}_j - k_{\ep}^2 {\vecu}_j |^2 +k_{\ep}^4 |{\vecu}_j|^2- k_{\ep}^2 |\curl  {\vecu}_j|^2 \, \, \dif x=0}.
\end{align*}
 Therefore, we have that for all $\vecu_j$ satisfy $\mathcal{A}_{k_{\ep}} ({\vecu_j}, {\vecu_j}) -k_{\ep}^2\mathcal{B}({\vecu_j},{\vecu_j}) \leq 0$. Again, using the fact that support of the  functions $\vecu_j$ are disjoint one can easily show that $\mathcal{A}_{k_{\ep}} ({\vecu_j}, {\vecu_i}) -k_{\ep}^2\mathcal{B}({\vecu_j},{\vecu_i}) = 0$ for all $i \neq j$. This implies that 
$$\mathcal{A}_{k_{\ep}} ({\vecu}, {\vecu}) -k_{\ep}^2\mathcal{B}({\vecu},{\vecu}) \leq 0 \quad \text{ for all}  \quad \vecu \in X_{M(\ep)}$$ 
and since $M(\ep) \to \infty$ as $\ep \to 0$ we have that there are infinitely many real transmission eigenvalues by appealing to Theorem 2.3 of \cite{chtevexist}. 
\end{proof}

\subsection{{\color{black}Dependence} on the parameters}
This section is dedicated to showing how the transmission eigenvalues depend on the material parameters $N$ and $\eta$. To this end, we will show that the transmission eigenvalues are monotonic with respect to the material parameters. Using the monotonicity we will then consider the case when $\eta$ tends to either zero or infinity. The case as $\eta$ tends to either zero can be handled just as in the scalar case see \cite{te-cbc2}. For the case when $\eta$ tends to infinity new analysis is given and that can be easily modified for the case of the scalar transmission eigenvalues. 

We now prove that the transmission eigenvalue depend monotonically on the material parameter. This gives that the transmission eigenvalues can be used to estimate one of the material parameters provided the other is known a prior. This can happen in the case of nondestructive testing where if we assume the scatterer is known but one wishes to determine changes to the interior of the material (i.e. changes in $N$).

We first recall that the transmission eigenvalues $k=k(N ,\eta)$ satisfy the equation 
\begin{equation}
\lambda_j (k;N ,\eta)-k^2=0 \label{teveq}
\end{equation}
where $\lambda_j $ is the $j$-th generalized eigenvalue defined in \eqref{geneig}. It is known that the positive generalized eigenvalues of \eqref{geneig} satisfy the min-max principle:
\begin{align*}
\lambda_j(k;N ,\eta)= \min\limits_{ U \in \mathcal{U}_j} \max\limits_{\vecu \in U\setminus \{ 0\} }  \frac{ \mathcal{A}_k(\vecu,\vecu) }{  \mathcal{B}(\vecu,\vecu)}  \,\, \text{ for } \, \, 1<n_{\text{min}}
\end{align*}
or
\begin{align*}
\lambda_j(k;N ,\eta) =\min\limits_{ U \in \mathcal{U}_j} \max\limits_{\vecu \in U \setminus \{ 0\} }  \frac{ \widetilde{\mathcal{A}}_k(\vecu,\vecu) }{ \widetilde{ \mathcal{B}}(\vecu,\vecu)}  \,\, \text{ for } \, \,n_{\text{max}}<1
\end{align*}
where $\mathcal{U}_j$ is the set of all $j$-dimensional subspaces of $X(D)$ whose {\color{black}intersection} with the null space of ${\mathbb{B}}$ or $\widetilde{\mathbb{B}}$ is trivial. 
Notice that the optimizer for the max-min principle when $k$ is a transmission eigenvalue is the corresponding eigenfunction. 

\begin{theorem}\label{mono}
{\color{black}Assume that for $\ell=1,2$ that $N_\ell$ and $\eta_\ell$ real-valued symmetric positive definite matrices such that for all $\xi \in \C^3$ 
$$N_1(x) \xi \cdot \ov{\xi} \leq N_2(x) \xi \cdot \ov{\xi} \,\, \,\text{ a.e. } x \in D \quad \text{and} \quad \eta_1(x) \xi \cdot \ov{\xi}  \leq \eta_2(x) \xi \cdot \ov{\xi} \,\,\, \text{ a.e. } x \in \partial D.$$
Then we have that
\begin{enumerate}
\item if $  |\xi|^2< N_1 \xi \cdot \ov{\xi}$, then $k_j(N_2 , \eta_2) \leq k_j(N_1 , \eta_1)$
\item if $N_2 \xi \cdot \ov{\xi} <  |\xi|^2$, then $k_j(N_1 , \eta_1) \leq k_j(N_2 , \eta_2)$
\end{enumerate}
where $k_j(N , \eta)$ is the smallest solution to \eqref{teveq} for any $j \in \N$. Moreover, if the inequalities for the parameters are strict, then the first transmission eigenvalue is strictly monotone. }
\end{theorem}
\begin{proof} 
First notice that by it's definition using the min-max principle we have that $\lambda_j (k)$ is continuous on $(0,\infty)$ for all $j \in \N$. We will prove the claim for the first case where $ |\xi|^2 < N_1 \xi \cdot \ov{\xi}$ for all $\xi \in \C^3$ and the other case is similar. Now, assume that $ |\xi|^2< N_1 \xi \cdot \ov{\xi}$ and it is clear by the definition of $ {\mathcal{A}}_k (\cdot \, ,\cdot)$ that 
$${\displaystyle {\mathcal{A}}_k (\vecu,\vecu) \big|^{\eta=\eta_2}_{N=N_2} \,  \leq \, {\mathcal{A}}_k (\vecu,\vecu)\big|^{\eta=\eta_1}_{N=N_1} }$$
for all $\vecu \in X(D)$ since for all $\xi \in \C^3$ 
$$\big(N_2(x)-I\big)^{-1} \xi \cdot \ov{\xi} \leq \big(N_1(x) -I\big)^{-1} \xi \cdot \ov{\xi} \,\, \,\text{ a.e. } x \in \overline{D} \quad \text{and} \quad \eta^{-1}_2(x) \xi \cdot \ov{\xi}  \leq \eta^{-1}_1(x) \xi \cdot \ov{\xi} \,\,\, \text{ a.e. } x \in \partial D.$$
This gives that for any $\vecu \in X(D) \setminus \text{Null}({\mathbb{B}})$ that 
$$\frac{ {\mathcal{A}}_k (\vecu,\vecu) \big|^{\eta=\eta_2}_{N=N_2} }{\mathcal{B} (\vecu,\vecu) }  \,  \leq \, \frac{  {\mathcal{A}}_k (\vecu,\vecu)\big|^{\eta=\eta_1}_{N=N_1} }{\mathcal{B} (\vecu,\vecu) }.$$
Therefore, by the min-max principle the above inequality yields that $\lambda_j(k ; N_2 , \eta_2) \leq \lambda_j(k; N_1 , \eta_1)$ for any $k$ positive. Now, by \eqref{teveq} we have that  $\lambda_j(k_1 ; N_1 , \eta_1)-k_1^2=0$ where $k_1$ is the transmission eigenvalue that is the smallest root of \eqref{teveq} corresponding to $N_1$ and $\eta_1$. Notice, that Theorem \ref{positive} implies that for all $k$ sufficiently small we have that $\lambda_j(k)-k^2>0$ for all $j\in \N$. By appealing to continuity for any $k$ positive we have that $\lambda_j(k ; N_2 , \eta_2) -k^2$ has at least one root in the interval $\left[ { c },  k_1 \right]$, for some $c>0$ and letting $k_2$ be the smallest root of $\lambda_j(k ; N_2 , \eta_2) -k^2$ we conclude that $k_2 \leq k_1$, proving the claim for this case. 

Now assuming that the inequalities for the parameters are strict then 
$${\displaystyle {\mathcal{A}}_k (\vecu,\vecu) \big|^{\eta=\eta_2}_{N=N_2} \,  < \, {\mathcal{A}}_k (\vecu,\vecu)\big|^{\eta=\eta_1}_{N=N_1} }$$
for all $\vecu \in X(D)$. Therefore, letting $\vecu_1$ be the transmission eigenfunction corresponding to the first real transmission eigenvalue $k_1(N_1 , \eta_1)$. We can then conclude that 
$$\lambda_1(k; N_2 , \eta_2) \leq \frac{ {\mathcal{A}}_k (\vecu_1,\vecu_1) \big|^{\eta=\eta_2}_{N=N_2} }{\mathcal{B} (\vecu_1,\vecu_1) }  \,  < \, \frac{  {\mathcal{A}}_k (\vecu_1,\vecu_1)\big|^{\eta=\eta_1}_{N=N_1} }{\mathcal{B} (\vecu_1,\vecu_1) } = \lambda_1(k; N_1 , \eta_1).$$
Then arguing just as above we obtain $k_1(N_2 , \eta_2) <k_1(N_1 , \eta_1)$. 
\end{proof}

From Theorem \ref{mono} we notice that assuming $N=nI$ and $\eta$ known we have that a constant $n$ can be uniquely determined by the first transmission eigenvalue. See for e.g. \cite{te-maxwell-fem} for estimating a constant parameter from the first transmission eigenvalue when $\eta=0$. For the case when $N$ is not a constant multiple of the identity then one tries to find a constant $n$ such that satisfies $k_1(nI , \eta) =k_1(N , \eta)$. This constant $n$ in computational examples is shown to approximately be that average of the eigenvalues of $N$.

Using the Monotonicity result and the variational formulation \eqref{varform} we can proceed just as in \cite{te-cbc2} for the scalar case to prove that as $\eta_{\text{max}} \to 0$ the transmission eigenvalues will converge to the classical Maxwell transmission eigenvalues (i.e. $\eta = 0$).  Moreover, the coercivity estimate gives that the eigenfunctions will converge to the classical Maxwell transmission eigenfunctions. We omit the proof to avoid repetition but it is a simple augmentation of the arguments for the scalar case.  
\begin{theorem} \label{conv-te-1}
There are infinitely many $(k_\eta , \vecu_\eta) \in \R^+ \times X(D)$ eigenpair satisfying \eqref{teprob3}--\eqref{teprob4} with $\eta \neq 0$ where as $\eta_{\text{max}} \to 0$ there is a subsequence such that $k_\eta \rightarrow k_0$ and $\vecu_\eta \to \vecu_0$ in $X(D)$ with $(k_0 , \vecu_0) \in \R^+ \times H_0(\text{curl}\,^2 ,D)$ eigenpair satisfying \eqref{teprob3}--\eqref{teprob4} with $\eta = 0$. 
\end{theorem}

We now study the case where the conductivity parameter tends to infinity. This case has not been studied for either the scalar or electromagnetic transmission eigenvalues. The numerical experiments in \cite{te-cbc} seem to suggest that as $\eta_{\text{min}} \to \infty$ the transmission eigenvalues for the scalar case will have a limit since they are monotone and bounded. Here we analyze this case for the electromagnetic transmission eigenvalues and similar analysis can be used for the scalar case. We will show that as $\eta_{\text{min}} \to \infty$ the transmission eigenvalues $k_\eta$ will converge to either a classical Maxwell eigenvalue or `Modified' Maxwell eigenvalue (defined below).

To begin, By appealing to Theorem \ref{mono} we can conclude that there are infinitely many transmission eigenvalues $k_\eta$ that are decreasing with respect to $\eta$. By Theorem \ref{positive} and \ref{mono} we have that there is a positive constant $c$ such that $c \leq k_\eta \leq k_0$. This implies that the set of transmission eigenvalues $\{ k_\eta \}$ is bounded with respect to $\eta$. Now consider a sequence of $\eta$ such that $\eta_{\text{min}} \to \infty$ and we now have that $k_\eta$ has a convergent subsequence where we let $k_\infty>0$ be the limit.  Also, notice that the corresponding transmission eigenfunction $\vecu_\eta$ is non-trivial and the coercivity estimate in Theorem \ref{Aforms} implies that we can take them to be normalized such that $\| \curl \vecu_\eta \|^2_{{\bf L}^2(D) } = 1$. By the coercivity estimate and the fact that $c \leq k_\eta \leq k_0$ we can conclude that there is a constant $\alpha >0$ independent of $\eta$ where 
$$\alpha \| \vecu_\eta \|^2_{X(D)} \leq \mathcal{A}_{\eta , k_\eta } (\vecu_\eta ,\vecu_\eta) =k_\eta^2\mathcal{B}(\vecu_\eta ,\vecu_\eta) \leq k^2_\eta \| \curl \vecu_\eta \|^2_{{\bf L}^2(D) } \leq k^2_0.$$
This gives that the sequence $\vecu_\eta$ is a bounded in the $X(D)$ norm and therefore has a weak limit $\vecu_\infty \in X(D)$. By Theorem \ref{varspace} we have that $ \curl \vecu_\eta$ is strongly convergent in ${\bf L}^2(D)$  which gives that $\| \curl \vecu_\infty \|^2_{{\bf L}^2(D) } = 1$ and therefore $\vecu_\infty \neq 0$.

 We will show that the limiting value $k^2_\infty$ is either a Maxwell eigenvalue or `Modified' Maxwell eigenvalue. Here we define $\tau \in \R_+$ as a Maxwell eigenvalue provided that there exists a nontrivial $\vecv \in H_0(\text{curl} ,D)$ where 
\begin{align}
\curl\curl\vecv - \tau \vecv=0  \quad \textrm{ in } \,  D. \label{maxwell-eig} 
\end{align}
Similarly, we say $\tau \in \R_+$ is a `Modified' Maxwell eigenvalue provided that there exists a nontrivial $\vecw \in H_0(\text{curl} ,D)$ where 
\begin{align}
\curl\curl \vecw - \tau N \vecw=0   \quad \textrm{ in } \,  D. \label{m-maxwell-eig} 
\end{align}
See \cite{AK-maxwell-book} for the existence of infinitely many  Maxwell and  `Modified' Maxwell eigenvalues.

Recall that the transmission eigenvalue problem \eqref{teprob3}--\eqref{teprob4} is equivalent to \eqref{teprob1}--\eqref{teprob2} where the eigenfunctions $\vecw_\eta$ and $\vecv_\eta$ are in ${\bf L}^2(D)$ and are defined by 
$$ k_\eta^2 \vecw_\eta =(N-I)^{-1} \big( \curl \curl \vecu_\eta -k_\eta^2 \vecu_\eta \big) \quad \text{and} \quad k_\eta^2 \vecv_\eta =(N-I)^{-1} \big( \curl \curl \vecu_\eta -k_\eta^2 N\vecu_\eta \big).$$ 
Since $\vecu_\eta$  is bounded in $X(D)$ we have that $\vecw_\eta$ and $\vecv_\eta$ are bounded ${\bf L}^2(D)$. This give that there exists subsequences (still denoted with $\eta$) such that $\vecw_\eta \weakc \vecw_\infty$ and $\vecv_\eta \weakc \vecv_\infty$ as $\eta_{\text{min}} \to \infty$ for some functions $\vecw_\infty$ and $\vecv_\infty$ in ${\bf L}^2(D)$. The goal is to prove that $\vecv_\infty$ and $\vecw_\infty$ satisfy either \eqref{maxwell-eig} or \eqref{m-maxwell-eig} respectively, and that both can't be zero vectors. Due to the fact that $\vecv_\eta$ and $\vecw_\eta$ satisfy \eqref{teprob1} it is clear that $\vecv_\infty$ and $\vecw_\infty$ satisfy \eqref{maxwell-eig} or \eqref{m-maxwell-eig} respectively in the distributional sense with $\tau=k^2_\infty$. Now, notice that the boundary condition \eqref{teprob2} implied that  
\begin{align*}
{\color{black} \eta^{-1} \big( ( \curl \vecu_\eta \big) \times \nu ) = \big( (\nu \times  \vecw_\eta ) \times \nu\big) \quad \text{ on } \,\,  \partial D}
\end{align*}
where the equality is understood in the trace sense. Since $\vecu_\eta$ is bounded in $X(D)$ Theorem \ref{varspace} gives that $\| (\curl \vecu_\eta) \times \nu \|^2_{{\bf L}^2(\partial D) } $ is bounded. 
{\color{black} By using the cross-product identity $|{\bf a} \times {\bf  b}|^2 = |{\bf a}|^2 |{\bf b}|^2 - ({\bf a} \cdot {\bf b})^2$ we obtain that }
$$ {\color{black} \| \vecw_\eta \times \nu \|^2_{{\bf L}^2(\partial D) } =\| (\nu \times \vecw_\eta )\times \nu \|^2_{{\bf L}^2(\partial D) }  \leq C \eta^{-2}_{\text{min}} \to 0 \quad \text{as} \quad \eta_{\text{min}} \to \infty}$$
where $C>0$ is a constant independent of $\eta$.  Therefore, by \eqref{teprob2} we can conclude that $\vecw_\infty \times \nu=\vecv_\infty \times \nu=0$ on $\partial D$. By appealing to the fact that
$$\curl\curl \vecw_\infty - k_\infty^2 N(x) \vecw_\infty=0 \quad \text{and} \quad  \curl\curl\vecv_\infty - k_\infty^2 \vecv_\infty=0  \quad \text{ in } \,  D$$
Green's Theorem gives that $\vecw_\infty$ and $\vecv_\infty$ are in $H_0(\text{curl} ,D)$. 

The one thing left to prove is that $\vecw_\infty$ and $\vecv_\infty$ can not both be the zero vector. To this end, assume on the contrary that they are both the zero vector. Then we can conclude by their representations with $\vecu_\infty$ that we have 
$$\curl\curl \vecu_\infty - k_\infty^2 N(x) \vecu_\infty=0 \quad \text{and} \quad  \curl\curl\vecu_\infty - k_\infty^2 \vecu_\infty=0  \quad \text{ in } \,  D.$$
Therefore, by subtracting the equations we conclude that $(N-I)\vecu_\infty =0$ in $D$ which gives that $\vecu_\infty$ is the zero vector since either $n_{\text{max}}<1$ or $n_{\text{min}}>1$ {\color{black} which implies that the matrix $N-I$ is invertible}. Due to the fact that $\| \curl \vecu_\infty \|^2_{{\bf L}^2(D) } = 1$ this gives a contradiction. Therefore, at least one of the limiting vectors $\vecv_\infty$ and $\vecw_\infty$ must be none trivial. This analysis implies the following result. 

\begin{theorem} \label{conv-te-2}
There are infinitely many  $(k_\eta , \vecv_\eta , \vecw_\eta) \in \R^+ \times{\bf L}^2(D) \times {\bf L}^2(D)$ eigenpairs satisfying \eqref{teprob1}--\eqref{teprob2} where as $\eta_{\text{min}} \to \infty$ there is a subsequence such that $k_\eta \rightarrow k_\infty$ where $k^2_\infty$ is either a Maxwell's eigenvalue or `Modified' Maxwell eigenvalue. Moreover, the corresponding eigenfunctions $\vecv_\eta \weakc \vecv_\infty$ a Maxwell's eigenfunction or $\vecw_\eta \weakc \vecw_\infty$ a `Modified' Maxwell eigenfunction as $\eta_{\text{min}} \to \infty$. 
\end{theorem}

\section{Scalar zero-index transmission eigenvalue problem}\label{acoustic-prob}
In this section, we derive the so-called zero-index transmission eigenvalue problem associated with \eqref{direct4}--\eqref{direct5} in a similar manner to the work done in \cite{zi-te}. The main advantage is that this is a linear eigenvalue problem which gives that we can appeal to standard analytical tools for studying eigenvalue problems for a compact operators. In \cite{zi-te} the authors derive the well-known  plate buckling eigenvalue problem for the case where $\eta =0$. This eigenvalue problem is derived by imbedding the scattering object $D$ in a background material with a zero index of refraction and studying the injectivity of the far-field operator. To this end, the total field is given by $\widetilde{u}=\widetilde{u}^s+u^i \in H^1_{loc}(\R^m)$ for {\color{black}$m=2,3$}  where again the incident field is defined by $u^i=\text{e}^{\text{i} k x \cdot d}$ with the incident direction $d$ satisfying 
\begin{align}
\Delta \widetilde{u} +k^2 \, \widetilde{u} =0  \quad \textrm{ in }  \R^m \setminus \overline{D}  \quad  \text{and} \quad  \Delta \widetilde{u} =0  \quad &\textrm{ in } \,  {D}  \label{direct6}\\
\widetilde{u} _+ - \widetilde{u} _-=0  \quad  \text{and} \quad {\partial_\nu \widetilde{u} _+}- {\partial_\nu \widetilde{u} _-} =0 \quad &\textrm{ on } \partial D. \label{direct7}
\end{align}
We assume that the scattered field $\widetilde{u}^s$ satisfies the Sommerfeld radiation condition uniformly with respect to the direction $\hat{x}=x/|x|$. It can be shown that \eqref{direct6}--\eqref{direct7} along with the radiation condition is well-posed (see for e.g. \cite{TE-book} for the mathematical framework). Since $D$ is known we can assume that the zero-index scattered field $\widetilde{u}^s$ is also known.

Similar to the previous section we assume that the coefficients in the scattering problem \eqref{direct4}--\eqref{direct5} are such that $n \in L^{\infty} (D)$ and $\eta \in L^{\infty} (\partial D)$. Furthermore, we assume that they are uniformly positive definite functions and there exists positive constants such that 
$$ n_{\text{min}}  \leq n(x) \leq n_{\text{max}}  \quad \text{ a.e.} \, \, x \in \overline{D} \quad \text{ and } \quad  \eta_{\text{min}} \leq \eta(x)\leq \eta_{\text{max}}\quad \text{ a.e.} \, \, x \in \partial D.$$
The scattered fields $u^s$ satisfying \eqref{direct4}--\eqref{direct5} and $\widetilde{u}^s$ satisfying \eqref{direct6}--\eqref{direct7} have the asymptotic expansion as  $|x| \to \infty$
$$u^s(x,d)=\frac{\text{e}^{\text{i}k|x|}}{|x|^{(m-1)/2}} \left\{u^{\infty}(\hat{x}, d ) + \mathcal{O} \left( \frac{1}{|x|}\right) \right\} $$
 {and} 
 $$ \widetilde{u}^s(x,d)=\frac{\text{e}^{\text{i}k|x|}}{|x|^{(m-1)/2}} \left\{\widetilde{u}^{\infty}(\hat{x}, d ) + \mathcal{O} \left( \frac{1}{|x|}\right) \right\}$$
(see for e.g. \cite{TE-book}). Here $u^{\infty}$ and $\widetilde{u}^{\infty}$ are the far-field patterns that depend on the incident and observation directions. 
Now we define the far-field operator ${F}$ and $\widetilde{F}: L^2(\mathbb{S}) \longmapsto  L^2(\mathbb{S})$ by 
$$Fg(\hat{x})=\int_{\mathbb{S}} u_\infty(\hat{x},d) g(d) \, \text{d}s(d) \quad \text{ and  } \quad \widetilde{F}g(\hat{x})=\int_{\mathbb{S}} \widetilde{u}_\infty(\hat{x},d) g(d) \, \text{d}s(d).$$
It is can be shown (see for e.g. \cite{TE-book}) using Rellich's Lemma and unique continuation that the relative far-field operator ${F}-\widetilde{F}$ is injective with a dense range provide that $k$ is not an associated zero-index transmission eigenvalue. These values are defined as the values $k \in \C$ such that there exists a nontrivial pair $(u,\widetilde{u} ) \in H^1(D) \times H^1(D)$ satisfying the system 
\begin{align}
\Delta u +k^2 n u=0 \quad \text{and} \quad \Delta \widetilde{u} =0  \quad &\textrm{ in } \,  D \label{zi-teprob1} \\
 u-\widetilde{u} =0  \quad  \text{and} \quad {\partial_\nu {u}} - {\partial_\nu \widetilde{u}}= \eta u   \quad &\textrm{ on } \partial D.  \label{zi-teprob2} 
\end{align} 

We will turn the zero-index transmission eigenvalue problem \eqref{zi-teprob1}--\eqref{zi-teprob2} into a fourth order eigenvalue problem. To this end, notice that the difference $w=u-\widetilde{u} \in H^1_0(D)$ satisfies 
$$\Delta w +k^2 n w= -k^2 n \widetilde{u} \quad \textrm{ in } \,  D$$
which implies that $w \in H^2(D) \cap H^1_0(D)$ by appealing to standard elliptic regularity(see for e.g. \cite{evans}). Using the fact that $\widetilde{u}$ is harmonic in $D$ along with the conductive boundary condition in \eqref{zi-teprob2} we have that 
\begin{align}
\Delta \frac{1}{n} (\Delta w +k^2 n w)=0 \quad \textrm{ in } \,  D \quad \text{and} \quad  {\partial_\nu w}= -  \frac{\eta}{k^2n}(\Delta w +k^2 n w)  \quad \textrm{ on } \,  \partial D. \label{zi-teprob3} 
\end{align} 
It is clear that the zero-index eigenvalue problems are equivalent since given a nontrivial solution $w\in H^2(D) \cap H^1_0(D)$ to \eqref{zi-teprob3} we can obtain $\widetilde{u} \in H^1(D)$ by solving 
$$\Delta \widetilde{u}=0  \quad \textrm{ in } \,  D \quad \text{ and } \quad   {\color{black} {\partial_\nu w} - \eta \widetilde{u} = 0}   \quad \textrm{ on } \partial D$$
since $ \eta^{-1} {\partial_\nu w}\in H^{1/2}(\partial D)$ and then defining $u=w + \widetilde{u}$. Therefore, we will study the zero-index transmission eigenvalue problem \eqref{zi-teprob3}. We will see that \eqref{zi-teprob3} is a linear eigenvalue problem for the spectral parameter $k^2$. Also, notice that the boundary value problem only requires $n>0$ where as the original transmission eigenvalue problem studied in \cite{te-cbc} one divides by $n-1$ which means that one must require that the contrast $n-1$ is of one sign in $D$. This can be  an onerous assumption to make for an unknown material parameter in practice.

\subsection{Existence of zero-index transmission eigenvalues}\label{exist-acoustic-prob}

In this section, we will prove the existence of infinity many real zero-index transmission eigenvalues. The eigenvalue problem \eqref{zi-teprob3} is easier to analyze than the standard transmission eigenvalue problem. Just as in the case of the electromagnetic eigenvalues studied in the previous section we will use a variational formulation to analyze \eqref{zi-teprob3}. To do so, we first recall that by the well-posedness of the Poisson problem for the Laplacian with zero Dirichlet data along with the $H^2$ elliptic regularity estimate (see for e.g. \cite{evans}) we have that there is a constant $C>0$ such that 
$$\|\varphi\|^2_{H^2(D)} \leq C \|\Delta \varphi\|^2_{L^2(D)} \quad \text{ for all } \quad \varphi \in H^2(D) \cap H^1_0(D).$$ 
This implies that we can take $\|\Delta \cdot \|_{L^2(D)}$ to be the norm with the associated inner-product on the Hilbert space $H^2(D) \cap H^1_0(D)$. Here $H^2(D)$ and $H_0^1(D)$ are the standard Sobolev spaces of $L^2$ functions with weak derivatives in $L^2$.

It is clear that by using Green's Theorem that the variational form of the eigenvalue problem \eqref{zi-teprob3} is given by
\begin{align}
\int\limits_{D} \frac{1}{n} \Delta w \Delta \overline{\varphi} -k^2 \grad w \cdot \grad \overline{\varphi} \, \text{d}x + \int\limits_{\partial D} \frac{k^2}{\eta} {\partial_\nu w} {\partial_\nu  \overline{\varphi} }\, \text{d}s= 0   \label{zi-varform} 
\end{align} 
for all $\varphi \in H^2(D) \cap H^1_0(D)$. Notice that in the variational formulation \eqref{zi-varform} we can use the Riesz representation theorem to define the bounded linear operators $\mathbb{T}$ and $\mathbb{K}$ that maps $H^2(D) \cap H^1_0(D)$ into itself such that 
$$(\mathbb{T} w, \varphi)_{H^2(D)} = \int\limits_{D} \frac{1}{n} \Delta w \Delta \overline{\varphi}\, \text{d}x$$
and 
$$(\mathbb{K} w, \varphi)_{H^2(D)} =\int\limits_{D} \grad w \cdot \grad \overline{\varphi} \, \text{d}x - {\color{black} \int\limits_{\partial D} \frac{1}{\eta} {\partial_\nu w} {\partial_\nu  \overline{\varphi} }\, \text{d}s }$$
for all $\varphi \in H^2(D) \cap H^1_0(D)$. This gives that $k$ is a zero-index transmission eigenvalue if and only if the null space of $\mathbb{T} - k^2 \mathbb{K}$ is non-trivial. From the definition we can clearly see that both $\mathbb{T}$ and $\mathbb{K}$ are self-adjoint since the sesquilinear forms defining them are Hermitian. The compact imbedding of  $H^{1/2}(\partial D)$ into $L^2(\partial D)$ and $H^{2}(D)$ into $H^1( D)$ implies that $\mathbb{K}$ is compact. Since $n_{\text{max}}>0$ we have that $\mathbb{T}$ is coercive and all together we have that there exists at most a discrete set of real valued $k^2$ where the null space of $\mathbb{T} - k^2 \mathbb{K}$ is non-trivial. Notice that for $k^2 \neq 0$ and a zero-index transmission eigenvalue then $1/k^2$ is an eigenvalue of the compact operator $\mathbb{T}^{-1}\mathbb{K}$. This gives the following result. 
\begin{theorem}  \label{zi-te-exists}
{\color{black}Assume that $n(x) \in L^{\infty} (D)$ and $\eta(x) \in L^{\infty} (\partial D)$ are uniformly positive. Then the set of zero-index transmission eigenvalues is countably infinite with no finite accumulation points. }
\end{theorem}

This does not give that the values $k$ are real since the operator  $\mathbb{K}$ is not necessarily positive due to the opposite signs in the definition so there may be negative eigenvalues which would correspond to $k$ being purely imaginary. Our numerical calculations have found purely imaginary eigenvalues $k$ which would suggest that the operator $\mathbb{K}$ is not positive. In order to prove the existence of infinitely many real zero-index transmission eigenvalues we again appeal to the analytic framework in \cite{chtevexist}. To do so, we know define the operators  $\mathbb{A}_k$ and $\mathbb{B}$ mapping $H^2(D) \cap H^1_0(D)$ into itself such that 
\begin{align}
(\mathbb{A}_k w, \varphi)_{H^2(D)} = \int\limits_{D} \frac{1}{n} \Delta w \Delta \overline{\varphi}\, \text{d}x +{\color{black} \int\limits_{\partial D} \frac{k^2}{\eta} {\partial_\nu w} {\partial_\nu  \overline{\varphi} }\, \text{d}s} \label{def-Ak}
\end{align} 
and 
\begin{align}
(\mathbb{B} w, \varphi)_{H^2(D)} =\int\limits_{D} \grad w \cdot \grad \overline{\varphi} \, \text{d}x \quad \text {for all} \quad \varphi \in H^2(D) \cap H^1_0(D) \label{def-B}
\end{align} 
via the Riesz representation theorem. This now give that $k$ is a zero-index transmission eigenvalue if and only if the null space of $\mathbb{A}_k - k^2 \mathbb{B}$ is non-trivial. Therefore, we have that by \eqref{def-B} $\mathbb{B}$ is a positive, self-adjoint compact operator by the compact imbedding of  $H^{2}(D)$ into $H^1( D)$. Clearly by \eqref{def-Ak} the operator $\mathbb{A}_k$ is  coercive, self-adjoint and depends analytically of $k$ with 
$$(\mathbb{A}_k w, w)_{H^2(D)} \geq n^{-1}_{\text{max}} \|\Delta w \|^2_{L^2(D)} \quad \text{ for all } \quad w \in H^2(D) \cap H^1_0(D).$$ 
This gives the following result. 

\begin{theorem}
{\color{black}Assume that $n(x) \in L^{\infty} (D)$ and $\eta(x) \in L^{\infty} (\partial D)$ are uniformly positive. Then $\mathbb{A}_k$ defined by  \eqref{def-Ak} is coercive, self-adjoint and $\mathbb{B}$ defined by \eqref{def-B} is a positive, self-adjoint and compact. }  
\end{theorem}

Just as in the previous section the above result gives that the operators  $\mathbb{A}_k$ and $\mathbb{B}$ satisfy the assumptions of Theorem 2.3 of \cite{chtevexist}. Therefore, to prove the existence result we need to show that the operator $\mathbb{A}_k -k^2 \mathbb{B}$ is positive on $H^2(D) \cap H^1_0(D)$ for some $k_1>0$ and is non-positive  on a subspace of $H^2(D) \cap H^1_0(D)$ for some $k_2$.

\begin{theorem}  \label{zi-te-exists}
{\color{black}Assume that $n(x) \in L^{\infty} (D)$ and $\eta(x) \in L^{\infty} (\partial D)$ are uniformly positive. Then there exists infinitely many real zero-index transmission eigenvalues.}
\end{theorem}
\begin{proof}
We begin by proving that for $k>0$ sufficiently small $\mathbb{A}_k -k^2 \mathbb{B}$ is a positive operator. Indeed, notice that by \eqref{def-Ak}--\eqref{def-B} we can conclude that 
$$\big( \mathbb{A}_kw-k^2\mathbb{B} w, w \big)_{H^2(D)} \geq \left({n^{-1}_{\text{max}}} - k^2 C\right) \|\Delta w \|^2_{L^2(D)}.$$
Here $C$ is the constant where $\|w\|^2_{H^2(D)} \leq C \|\Delta w\|^2_{L^2(D)}$ given by the $H^2$ elliptic regularity estimate and Wellposedness of the Poisson problem. This gives positivity for all $k>0$ such that $k^2< {n^{-1}_{\text{max}}}C^{-1}$.

Now we construct a subspace and find a value of $k>0$ where the operator is non-positive. To this end, let $B_j=B(x_j , \ep):=\{ x \in \R^3 : |x-x_j | <  \ep \}$ where $x_j \in D$ and $\ep>0$. Define $M(\ep)$ the supremum of the number of disjoint balls $B_j$ such that $\overline{B_j} \subset D$. It is well known that there are infinitely many values $k_\ep$ that correspond to the plate buckling eigenvalue problem 
\begin{align*}
\Delta \frac{1}{n_{\text{min}}} \Delta w_j = - k_\ep^2 \Delta w_j \quad \textrm{ in } \,  B_j \quad \text{ where } \,\, w_j \in H^2_0( B_j ).
\end{align*}
{\color{black}where the Hilbert space $H^2_0( B_j ) =  \{ u \in H^2(B_j) \, : \, u = \partial_\nu u =0 \,\, \text{ on } \,\, \partial B_j\}.$}
We extend $w_j$ into $D$ such that $w_j = 0$ in $D \setminus \overline{B_j}$. It is clear that $X_{M(\ep)} = \text{ span}\{ {w}_1 , {w}_2 , \cdots,  {w}_{M(\ep)}  \}$ is a $M(\ep)$ dimensional subspace of $H^2(D) \cap H^1_0(D)$ since the basis functions are orthogonal. Using the variational formulation of the plate buckling eigenvalue problem we have that 
\begin{align*}
\big( \mathbb{A}_{k_\ep}w_j -k_\ep^2\mathbb{B} w_j, w_j  \big)_{H^2(D)}  &= \int\limits_{D} \frac{1}{n} |\Delta w_j|^2 \ -k_\ep^2 |\grad w_j|^2  \, \text{d}x + {\color{black}\int\limits_{\partial D} \frac{k_\ep^2}{\eta} |{\partial_\nu w_j}|^2\, \text{d}s}\\
										    &\leq \int\limits_{B_j} \frac{1}{n_{\text{min}}} |\Delta w_j|^2 \ -k_\ep^2 |\grad w_j|^2  \, \text{d}x =0.
\end{align*} 
It is clear that do to the disjoint support that {\color{black} $\big( \mathbb{A}_{k_{\ep}}w_j-k_\ep^2\mathbb{B} w_j, w_i \big)_{H^2(D)} = 0$ for all $i \neq j$ which implies that }
$$\big( \mathbb{A}_{k_{\ep}}w-k_\ep^2\mathbb{B} w, w \big)_{H^2(D)} \leq 0 \quad \text{ for all } \quad w \in X_{M(\ep)}$$
and since $M(\ep) \to \infty$ as $\ep \to 0$ we have that there are infinitely many real zero-index transmission eigenvalues by Theorem 2.3 of \cite{chtevexist}. 
\end{proof}

\subsection{{\color{black}Dependence} on the parameters}
We now turn our attention to studying how the zero-index transmission eigenvalues depend on the coefficients $n$ and $\eta$. Just as in the previous section we will show that the eigenvalues are monotone with respect to the coefficients and using the monotonicity result we will then analyze the case when $\eta$ tends to either zero or infinity. To prove that the zero-index transmission eigenvalue depend monotonically on the coefficients we will proceed just as in the case for the transmission eigenvalues. Therefore, notice that the zero-index transmission eigenvalues $k=k(n ,\eta)$ satisfy the equation 
\begin{equation}
\lambda_j (k;n ,\eta)-k^2=0 \label{zi-teveq}
\end{equation}
where $\lambda_j $ is the $j$-th generalized eigenvalue for the operators $\mathbb{A}_k$ and $\mathbb{B}$ which are defined by equations \eqref{def-Ak}--\eqref{def-B} such that 
$$ \mathbb{A}_k w= \lambda_j \mathbb{B} w. $$
By the variational definition in \eqref{def-Ak} it is clear that $\lambda_j(k)$ depends continuously on $k \in (0,\infty)$ and satisfy the min-max principle:
\begin{align*}
\lambda_j(k;n ,\eta)= \min\limits_{ U \in \mathcal{U}_j} \max\limits_{  w \in U\setminus \{ 0\} }  \frac{ (\mathbb{A}_k w,w)_{H^2(D)} }{  (\mathbb{B}w,w)_{H^2(D)} }  \,\, \text{ for } \, \, 0<n_{\text{min}}
\end{align*}
where $\mathcal{U}_j$ is the set of all $j$-dimensional subspaces of $H^2(D) \cap H^1_0(D)$. Following in the same way as in Theorem \ref{mono} we have the following result. 

\begin{theorem}\label{mono2}
{\color{black}Assume that for $\ell=1,2$ that $n_\ell$ and $\eta_\ell$ real-valued uniformly positive definite function such that $n_1(x) \leq n_2(x)$  for a.e. $x \in D$ and $\eta_1(x)  \leq \eta_2(x)$ for a.e. $x \in \partial D$.
Then $k_j(n_2 , \eta_2) \leq k_j(n_1 , \eta_1)$ where $k_j$ denotes the smallest solution to \eqref{zi-teveq} for any $j \in \N$. Moreover, if the inequalities for the parameters are strict, then the first zero-index transmission eigenvalue is strictly monotone. }
\end{theorem}

Similarly, using the above monotonicity result and the variational formulation \eqref{zi-varform} one can argue just as in \cite{te-cbc2} for the scalar transmission eigenvalues with a coated boundary condition that as $\eta_{\text{max}} \to 0$ the zero-index transmission eigenvalues and eigenfunctions will converge to the classical plate buckling eigenvalues and eigenfunctions. The monotonicity also gives that information about the refractive index can be recovered from the eigenvalues assuming $\eta$ is known. Clearly, a constant refractive index can be uniquely recovered from the eigenvalues and for a non-constant refractive index can be estimated the the constant that best approximates its first zero-index transmission eigenvalues. We now state the convergence result as the boundary parameter $\eta$ tends to zero. 

\begin{theorem} \label{conv-te-3}
There are infinitely many $(k_\eta , w_\eta) \in \R^+ \times H^2(D) \cap H^1_0(D)$ eigenpair satisfying \eqref{zi-teprob3} for $\eta >0$ where as $\eta_{\text{max}} \to 0$ there is a subsequence such that $k_\eta \rightarrow k_0$ and $w_\eta \to w_0$ in $H^2(D)$ with $(k_0 , w_0) \in \R^+ \times H^2_0(D)$ being an eigenpair satisfying \eqref{zi-teprob3} with $\eta = 0$ (i.e. the plate buckling eigenvalue problem). 
\end{theorem}

We now analyze the zero-index transmission eigenvalue problem as ${\color{black}\eta_{\text{min}} \to \infty}$. Just as in the previous section we will see that the zero-index transmission eigenvalues and eigenfunctions will have limits that converge to a corresponding eigenpair. Notice that Theorem \ref{mono2} and the proof of Theorem \ref{zi-te-exists} implies that there exists infinitely many eigenvalues $k_\eta$ that for some $\delta >0$ satisfy $\delta \leq k_\eta \leq k_0$ where $k_0$ is a plate buckling eigenvalue. This implies that as ${\color{black} \eta_{\text{min}} \to \infty}$ there must be a limit point denoted $k_\infty >0$.  

Now assume that the corresponding eigenfunction $w_\eta$ is normalized such that $\| w_\eta \|_{H^1(D) } = 1$. Using the variational formulation \eqref{zi-varform} we can conclude that 
$$\frac{1}{n_{\text{min}}} \| \Delta w_\eta \|^2_{L^2(D)} \leq  k^2_\eta \| \grad w_\eta \|^2_{{ L}^2(D) }$$
which gives that $w_\eta \weakc w_\infty$ as  ${\color{black} \eta_{\text{min}} \to \infty}$ to some $w_\infty \in H^2(D) \cap H^1_0(D)$. The compact imbedding of $H^2(D)$ into $H^1(D)$ gives that $w_\eta \to w_\infty$ in $H^1(D)$ which implies that $w_\infty$ has unit norm in $H^1(D)$. Now let $\widetilde{u}_\eta$ be the zero-index transmission eigenfunction determined by solving 
$$\Delta \widetilde{u}_\eta=0  \quad \textrm{ in } \,  D \quad \text{ and } \quad {\color{black} \widetilde{u}_\eta = \frac{1}{\eta} {\partial_\nu w_\eta} } \quad \textrm{ on } \partial D.$$
Since $w_\eta$ is bounded in $H^2(D) \cap H^1_0(D)$ the Trace Theorem implies that ${\partial_\nu w_\eta}$ is bounded in $H^{1/2}(\partial D)$. Therefore, by the well-posedness estimate for the Dirichlet problem we have that $\| \widetilde{u}_\eta \|_{H^1(D) } \leq C \eta^{-1} _{\text{min}} \to 0$ as ${\color{black} \eta_{\text{min}} \to \infty}$. Noticing that 
$$\Delta w_\eta +k_\eta^2 n w_\eta= -k_\eta^2 n \widetilde{u}_\eta \quad \textrm{ in } \,  D$$
we have that 
$$\Delta w_\infty +k_\infty^2 n w_\infty= 0 \quad \textrm{ in } \,  D \quad \text{ and } \quad \| w_\infty \|_{H^1(D) } = 1.$$
This gives that $k^2_\infty$ is a `Modified' Dirichlet eigenvalue for the Laplacian and $w_\infty$ is the corresponding `Modified' Dirichlet eigenfunction. We will now show that the convergence of the eigenfunctions is strong in $H^2(D) \cap H^1_0(D)$. Indeed, notice that we have 
$$\Delta (w_\infty - w_\eta) = (k^2_\eta -k^2_\infty) n w_\infty + k^2_\eta n(w_\eta - w_\infty ) + k^2_\eta n \widetilde{u}_\eta$$
which implies that $\| \Delta (w_\infty - w_\eta) \|_{{ L}^2(D) } \to 0$ as ${\color{black} \eta_{\text{min}} \to \infty}$, proving the convergence. From the above analysis we have the following convergence result.

\begin{theorem} \label{conv-te-4}
There are infinitely many  $(k_\eta , w_\eta) \in \R^+ \times H^2(D) \cap H^1_0(D)$ eigenpairs satisfying \eqref{zi-teprob3} where as $\eta_{\text{min}} \to \infty$ there is a subsequence such that $k_\eta \rightarrow k_\infty$ where $k^2_\infty$ is a `Modified' Dirichlet eigenvalue and  $w_\eta \to w_\infty$ the corresponding `Modified' Dirichlet eigenfunction. 
\end{theorem}

\subsection{Numerical examples for the unit ball}
In this section, we provided some numerical examples for the zero-index transmission eigenvalue problem. The domain $D$ is assumed to be the unit ball in $\R^2$ where we use separation of variables for constant material parameters and a Galerkin method for variable material parameters to approximate the eigenvalues. We will examine the monotonicity of the eigenvalues with respect to the material parameters as well as the convergence results in Theorems \ref{conv-te-3} and \ref{conv-te-4} with respect to $\eta$. Using the monotonicity result for the refractive index $n$ one can estimate the refractive index (see for e.g. \cite{spectraltev},  \cite{te-homog}, and \cite{ghtevpaper}). The Galerkin method we employ here is similar to the work done in \cite{gpinvtev} where we pick a finite dimensional subspace of $H^2(D) \cap H^1_0(D)$ associated with a differential operator that becomes dense as the dimension tends to infinity.

All of the numerical examples presented are done using MATLAB 2018a on an iMac with a 4.2GHz Intel Core i7 processor with 8GB of memory. We now assume that both $n$ and $\eta$ are constants. Therefore, since $D$ is the unit ball we have that the eigenfunctions $u$ and $\widetilde{u}$ satisfying the PDEs in \eqref{zi-teprob1} have the series representation 
$$ u(r,\theta)=\sum_{|m|=0}^\infty \alpha_m J_{|m|}(k\sqrt{n}r) \text{e}^{\text{i}m\theta}  \quad \text{ and } \quad \widetilde{u}(r,\theta)=\sum_{|m|=0}^\infty \beta_m r^{|m|} \text{e}^{\text{i}m\theta}$$
for $m \in \Z$ where $J_{|m|}$ denotes the Bessel function of the first kind. By applying the boundary conditions in \eqref{zi-teprob2} we obtain a homogeneous linear system of equations for $\alpha_m$ and $\beta_m$ which has a non-trivial solution if and only if for some $m$ we have $d_m(k)=0$ where we define 
$$d_m(k) = k\sqrt{n}J_{|m|}' \big(k\sqrt{n}\big) -\big(\eta +|m|\big)J_{|m|}\big(k\sqrt{n}\big) \quad \text{ for all } \quad m \in \Z.$$
This gives that we can compute the zero-index transmission eigenvalues by finding the roots of the transcendental function $d_m(k)$. Notice that from the definition of $d_m(k)$ we have that at an eigenvalue 
$$J_{|m|} \big(k\sqrt{n}\big) = \frac{  k\sqrt{n} }{\big(\eta +|m|\big)} J_{|m|}'\big(k\sqrt{n}\big)$$
and we can clearly verify that $k_\eta$ will converge to a root of $J_{|m|} \big(t\sqrt{n}\big)$ as $\eta \to \infty$ which corresponds to the `Modified' Dirichlet eigenvalues. This give that all the zero-index transmission eigenvalues for the unit ball with constant coefficients will converge to a `Modified' Dirichlet eigenvalue as $\eta \to \infty$. We can compute the roots of $d_m(k)$ by using the built in root finding function in MATLAB `fzero' where the initial guess is determined by graphing the function. 

We now show the monotonicity and convergence of the zero-index transmission eigenvalues as $\eta$ tend to either zero or infinity. First we show that as $\eta \to \infty$ the zero-index transmission eigenvalue will converge to a `Modified' Dirichlet eigenvalue. In Table \ref{monotable1} we see the convergence of the first root of the function $d_0 (k)$ to the first `Modified' Dirichlet eigenvalue for $n=4$ which is given by $k_\infty = 1.2024$ as well as the monotonicity which gives that the zero-index transmission eigenvalues are decreasing with respect to $\eta$. To compute the rate of convergence (ROC) denoted by $p$ we assume that 
$$\big| k(\eta)-k_\infty \big| \approx C \eta^{-p} \quad \text{ which implies that } \quad  \log \big( \big| k(\eta)-k_\infty \big|  \big) \approx \log(C) - p \log( \eta )$$
for some constant $C$ independent of $\eta$. In our calculations we compute the convergence rate to be approximately first order. This validates the monotonicity and convergence discussed in the previous section. 
\begin{table}[ht!]
\centering  
\begin{tabular}{c c c } 
$\eta =10^{j}$  $\quad$   &  $k(\eta)$  $\quad$ & ROC $p$   \\ [0.5ex] 
\hline                  
$j=$0 $\quad$ \vline& 1.8499 $\quad$ \vline&  --  \\
$j=$1 $\quad$ \vline& 1.4435 $\quad$ \vline&  0.4290  \\ 
$j=$2 $\quad$ \vline& 1.2267 $\quad$ \vline&  0.9966  \\ 
$j=$3 $\quad$ \vline& 1.2048 $\quad$ \vline&  1.0054  \\ 
$j=$4 $\quad$ \vline& 1.2027 $\quad$ \vline&  0.9031  \\ 
\hline 
\end{tabular}
\caption{Monotonicity and Convergence as $\eta \to \infty$ for $n=4$ where the limit $k_\infty = 1.2024$.}\label{monotable1}
\end{table}
Likewise, in Table \ref{monotable2} we check the convergence of the first root of the function $d_0 (k)$ as $\eta \to 0$ where the limit $k_0$ is a plate buckling eigenvalue. We also wish to check the rate of convergence without having to compute the plate buckling eigenvalues. Therefore, we assume that 
$$\big| k(\eta)-k_0 \big| \approx C \eta^{p} \quad \text{ which implies that } \quad   \frac{\big| k(\eta)-k\big(\eta/2\big) \big|}{\big| k\big(\eta/2\big)-k\big(\eta/4\big) \big|} \approx 2^p $$
where again the constant $C$ is independent of $\eta$. Using this we can estimate convergence rate which our calculations shows to be first order. Also, in Table \ref{monotable2} we see that as $\eta$ decreases the zero-index transmission eigenvalue increasing just as the theory predicts.  
\begin{table}[ht!]
\centering  
\begin{tabular}{c c c } 
$\eta ={1}/{2^{j}} \quad $    &  $k(\eta)$ $\quad$  & ROC  $p$  \\ [0.5ex] 
\hline                  
$j=$0 $\quad$ \vline& 1.8499 $\quad$ \vline&  --  \\
$j=$1 $\quad$ \vline& 1.8830 $\quad$ \vline&  --  \\ 
$j=$2 $\quad$ \vline& 1.8995 $\quad$ \vline&  --  \\ 
$j=$3 $\quad$ \vline& 1.9077 $\quad$ \vline&  1.0044  \\ 
$j=$4 $\quad$ \vline& 1.9118 $\quad$ \vline&  1.0088  \\ 
$j=$5 $\quad$ \vline& 1.9130 $\quad$ \vline&  1.0356  \\ 
$j=$6 $\quad$ \vline& 1.9148 $\quad$ \vline&  1.0000  \\ 
\hline 
\end{tabular}
\caption{Monotonicity and Convergence as $\eta \to 0$ for $n=4$.}\label{monotable2}
\end{table}

Now we present a numerical method for computing the zero-index transmission eigenvalues. Our method is a based on a Spectral-Galerkin approximation method for variational formulation \eqref{zi-varform}. Therefore, we assume that the eigenfunction $w$ satisfying \eqref{zi-varform} has a series representation with respect to some basis of $H^2(D) \cap H^1_0(D)$ and the approximation is given by considering the truncated series in the finite dimensional subspace spanned by the first $M \in \N$ basis functions. 

In \cite{ghtevpaper} the standard transmission eigenvalue problem was considered with $\eta=0$ where the authors compute the eigenvalues using a Spectral-Galerkin method. For the standard transmission eigenvalue problem the eigenfunctions are in $H^2_0(D)$ and the basis functions used are the eigenfunctions for the Bilaplacian. Here we take our basis to be the Dirichlet eigenfunctions for the Laplacian. To this end, let $\phi_j$ be the Dirichlet eigenfunction such that 
\begin{align}
- \Delta \phi_i = \lambda_i \phi_i \,\, \text{ in } \,\, D \quad \text{ where }  \quad  \phi_i \in H^1_0(D) \label{dirichlet-eig}
\end{align}
where $\lambda_i >0$ is the corresponding Dirichlet eigenvalue. By elliptic regularity we have that $\phi_i \in H^2(D) \cap H^1_0(D)$  and we have that $\{ \phi_i \}_{i=1}^\infty$ forms an orthogonal set due to the $L^2(D)$ orthogonality. We now show that the Dirichlet eigenfunctions form a complete orthogonal set for $H^2(D) \cap H^1_0(D)$. 

\begin{theorem} \label{dense-eig}
Let $\phi_i$ satisfy \eqref{dirichlet-eig}  then the span of $\{ \phi_i \}_{i=1}^\infty$ is dense in $H^2(D) \cap H^1_0(D)$. 
\end{theorem}
\begin{proof}
We take the norm on $H^2(D) \cap H^1_0(D)$ to be given by $\|\Delta \cdot \|_{L^2(D)}$ with the associated inner-product. Now, we let $f \in H^2(D) \cap H^1_0(D)$ be orthogonal to $\phi_i$ for all $i \in \N$. Therefore, by appealing to Green's Theorem and equation \eqref{dirichlet-eig} we have that 
\begin{align*}
0 = \int\limits_D  \Delta \phi_i  \Delta \ov{f} \, \text{d}x = - \lambda_i \int\limits_D \phi_i  \Delta \ov{f}  \, \text{d}x = - \lambda_i \int\limits_D \ov{f} \Delta \phi_i \, \text{d}x = \lambda^2_i \int\limits_D \phi_i  \ov{f}  \, \text{d}x.
\end{align*}
This implies that $f$ is orthogonal to $\phi_i$ for all $i \in \N$ with respect to the $L^2(D)$ inner-product. Since $\{ \phi_i \}_{i=1}^\infty$ is an orthogonal basis for $L^2(D)$ this implies that $f=0$, proving the claim. 
\end{proof}

Using Theorem \ref{dense-eig} we have that the zero-index transmission eigenfunctions $w$ satisfying \eqref{zi-varform} can be written as the $H^2(D) \cap H^1_0(D)$ convergent series 
$$w(x) = \sum\limits_{i=1}^{\infty} w_i \phi_i (x) \quad \text{ with complex-valued constants } \,\, w_i $$
where $\phi_i$ satisfy \eqref{dirichlet-eig}. We approximate the eigenfunction by the truncated series 
$$w^{(M)} (x) = \sum\limits_{i=1}^{M} w_i \phi_i (x) \quad \text{ for some fixed } \,\, M  \in \N.$$
Due to the density of the $\phi_i$ it is clear that $\| \Delta (w^{(M)} -  w) \|_{L^2(D)} \to 0$ as $M \to \infty$ (see for e.g. \cite{numerics-book}). We substitute $w^{(M)}$ into the variational formulation \eqref{zi-varform} for the test function $\phi_j$ to obtain a Dirichlet-Spectral approximation of the zero-index transmission eigenvalue problem. This gives that the approximated eigenvalues $k^{(M)}$ satisfy the linear matrix eigenvalue problem
\begin{align}
\left( {\bf A}-\big( k^{(M)}\big)^2{\bf B}\right) \vec{w} =0 \quad \text{ where } \quad \vec{w} \neq 0. \label{g-eig}
\end{align}
We have that the $M \times M$ matrices in the Dirichlet-Spectral approximation of \eqref{zi-varform} are given by
$${\bf A}_{i,j}  = \lambda_i \lambda_j \int\limits_{D} \frac{1}{n(x)}  \phi_i (x) \, \ov{\phi}_j(x) \, \text{d}x$$ 
and
$${\bf B}_{i,j} = \lambda_i \int\limits_{D} \phi_i (x) \,  \ov{\phi}_j(x) \, \text{d}x - \int\limits_{\partial D} \frac{1}{\eta(x) }\,  {\partial_\nu \phi_i (x)}  \, {\partial_\nu  \ov{\phi}_j (x) }\, \text{d}s.$$
where we have used \eqref{dirichlet-eig} along with Green's Theorem to obtain that 
$$\int\limits_{D} \frac{1}{n(x)} \Delta  \phi_i (x) \, \Delta\ov{\phi}_j(x) \, \text{d}x =  \lambda_i \lambda_j \int\limits_{D} \frac{1}{n(x)}  \phi_i (x) \, \ov{\phi}_j(x) \, \text{d}x$$
and 
$$ \int\limits_{D} \grad \phi_i (x) \cdot \grad  \ov{\phi}_j(x) \, \text{d}x =  \lambda_i \int\limits_{D} \phi_i (x) \,  \ov{\phi}_j(x) \, \text{d}x.$$ 
We have that the discretized eigenvalue problem is a selfadjoint generalized eigenvalue problem since the coefficients are real valued. Standard arguments pertaining to the convergence of compact operators \cite{osborn} give the following convergence result. 
\begin{theorem} \label{discrete-eig-conv}
There exists $M$ eigenvalues satisfy \eqref{g-eig} and for each fixed $j$ we have that $k_j^{(M)} \to k_j$ as $M \to \infty$ where $k_j$ is a zero-index transmission eigenvalue satisfying \eqref{zi-varform}. 
\end{theorem}

We are now ready to compute the zero-index transmission eigenvalues using the Dirichlet-Spectral approximation. Since $D$ is the unit circle in $\R^2$ we have that the Dirichlet eigenfunctions and eigenvalues for the Laplacian are given by 
$$\phi_{p,q}(r,\theta) = J_{p}(\tau_{p,q} \, r) \text{e}^{\text{i}p\theta} \quad \text{ and } \quad \lambda_{p,q} = \tau_{p,q}^2$$
where $\tau_{p,q}$ is the $q-$th positive root of $p-$th Bessel function of the first kind for all $p\in \N \cup\{0\}$ and $q \in \N$. In our calculations we take $M=16$ which will correspond to $0\leq p\leq 3$ and $1\leq q \leq 4$ which gives $16 \times 16$ matrices for the generalized eigenvalue problem \eqref{g-eig}. The matrices ${\bf A}$ and ${\bf B}$ are computed using a Gaussian quadrature method where the integrals are written in polar coordinates. The eigenvalues for \eqref{g-eig} are computed using the `eig' function in MATLAB. In order to assure that the method is accurately approximating the zero-index transmission eigenvalues we check the Dirichlet-Spectral approximation v.s. Analytic values given by separation of variables for $n=4$ and $\eta = 1$ which is presented in Table \ref{eig-compare}. In our calculations we also see that  $k = 1.2465\text{i}$, $-1.5596\text{i}$ and $-2.0255\text{i}$ are zero-index transmission eigenvalues for $n=4$ and $\eta = 1$.
\begin{table}[ht!]
\centering  
\begin{tabular}{c c c } 
Approximation $\quad\quad$    &  Analytic  $\quad\quad$ & Relative error    \\ [0.5ex] 
\hline                  
$k_1 = 1.8743 \quad$ \vline& $k_1 = 1.8499 \quad$ \vline&  $0.0132$ \\
$k_2 = 2.5860 \quad$ \vline& $k_2 = 2.5678 \quad$ \vline&  $0.0071$ \\
$k_3 = 3.2481 \quad$ \vline& $k_3 = 3.2299 \quad$ \vline&  $0.0056$ \\
\hline 
\end{tabular}
\caption{Comparison of the Dirichlet-Spectral approximation v.s. Analytic values from separation of variables for $n=4$ and $\eta = 1$ for the first three zero-index transmission eigenvalues. }\label{eig-compare}
\end{table}

We now present some examples with variable coefficients where we check the monotonicity result given in the previous section. Table \ref{eig-compare} shows that our approximation method does compute the eigenvalues with some precision. We now compute the eigenvalues for variable valued refractive index and boundary parameter. To this end, we first check the monotonicity with respect to the refractive index for $\eta =1$ where we take 
$$n_1=4-r^2\left(1-\frac{1}{2}\sin(\theta)\right) \quad \text{ and } \quad n_2=4+r^2\left(1-\frac{1}{2}\sin(\theta)\right).$$
From Theorem \ref{mono2} we have that the $k_j(n_2) \leq k_j(4) \leq k_j(n_1)$ where $k_j (n)$ denotes the zero-index transmission eigenvalues for $n$ with $\eta =1$ is fixed. In Table \ref{eig-mono1} we report the first three eigenvalues for $n_1$ and $n_2$ given by our Dirichlet-Spectral approximation. 
\begin{table}[ht!]
\centering  
\begin{tabular}{c c c } 
 $n_1  \quad\quad$    &  $n=4  \quad\quad$ & $n_2$   \\ [0.5ex] 
\hline                  
$k_1 = 1.9336 \quad$ \vline& $k_1 = 1.8743 \quad$ \vline&  $k_1 = 1.8244$ \\
$k_2 = 2.6747 \quad$ \vline& $k_2 = 2.5860 \quad$ \vline&  $k_2 = 2.5067$ \\
$k_3 = 3.3803 \quad$ \vline& $k_3 = 3.2481 \quad$ \vline&  $k_3 = 3.1327$ \\
\hline 
\end{tabular}
\caption{The first three zero-index transmission eigenvalues for different refractive indices. Here we see the monotonicity with respect to the refractive index for $\eta =1$. }\label{eig-mono1}
\end{table}
Now we check the monotonicity of the  zero-index transmission eigenvalues with respect to the boundary parameter $\eta$. In Table \ref{eig-mono2} we report the first three eigenvalues given by our  approximation with the refractive index $n=4$ where 
$$\eta_1=\frac{1}{1+2\sin^2(\theta)} \quad \text{ and } \quad \eta_2=1+2\sin^2(\theta).$$
Again, Theorem \ref{mono2} gives that $k_j(\eta_2) \leq k_j(1) \leq k_j(\eta_1)$ where $k_j (\eta)$ denotes the zero-index transmission eigenvalues for $\eta$ with $n =4$ is fixed. 
\begin{table}[ht!]
\centering  
\begin{tabular}{c c c } 
 $\eta_1  \quad\quad$    &  $\eta=1  \quad\quad$ & $\eta_2$   \\ [0.5ex] 
\hline                  
$k_1 = 1.9428 \quad$ \vline& $k_1 = 1.8743 \quad$ \vline&  $k_1 = 1.7565$ \\
$k_2 = 2.6375 \quad$ \vline& $k_2 = 2.5861 \quad$ \vline&   $k_2 = 2.5020$ \\
$k_3 = 3.2901 \quad$ \vline& $k_3 = 3.2481 \quad$ \vline&  $k_3 = 3.1691$ \\
\hline 
\end{tabular}
\caption{The first three zero-index transmission eigenvalues for different boundary parameters. Here we see the monotonicity with respect to the boundary parameters for $n =4$.}\label{eig-mono2}
\end{table}
In this example, we compute the zero-index transmission eigenvalue when both $n$ and $\eta$ are non-constant. Just as in the previous examples we see the monotonicity where 
$$k_j(n_2 , \eta_2)  \leq k_j(4,1) \leq k_j(n_1 , \eta_1)$$
where $k_j (n , \eta)$ denotes the zero-index transmission eigenvalues for $n$ and $\eta$.  In Table \ref{eig-mono3} the eigenvalues for the case where both $n$ and $\eta$ are non-constant are reported. 
\begin{table}[ht!]
\centering  
\begin{tabular}{c c c } 
$n_1$ , $\eta_1  \quad\quad$  &  $n=4$ , $\eta=1  \quad\quad$  & $n_2$ , $\eta_2$   \\ [0.5ex] 
\hline                  
$k_1 = 2.0119 \quad$ \vline& $k_1 = 1.8743 \quad$ \vline&  $k_1 = 1.7168$ \\
$k_2 = 2.7329 \quad$ \vline& $k_2 = 2.5860 \quad$ \vline&   $k_2 = 2.4288$ \\
$k_3 = 3.4269 \quad$ \vline& $k_3 = 3.2481 \quad$ \vline&  $k_3 = 3.0587$ \\
\hline 
\end{tabular}
\caption{The first three zero-index transmission eigenvalues for both $n$ and $\eta$ are non-constant.}\label{eig-mono3}
\end{table}
We also consider approximating the refractive index provided $\eta$ is known. To this end, we want to find a constant $n_{\text{approx}}$ such that $k_1(n_{\text{approx}} ) =  k_1(n)$ for the constant refractive index $n=4$ as well as the two given variable refractive indices $n_1$ and $n_2$. In these calculations $\eta=1$ and we approximate the eigenvalue $k_1$ using our approximation for constant $n \in [2,8]$ which is given in Figure \ref{monoplot}. In order to compute the approximation of the refractive index $n_{\text{approx}}$ we find the polynomial interpolation for $k_1(n)$ via the `polyfit' command in MATLAB then `fzero' is used to solve for the approximate. See {\color{black} Figures \ref{monoplot} and \ref{reconplot}} for the estimated $n_{\text{approx}}$. Here we see that $n_{\text{approx}} \approx n$ for the constant refractive index and for the variable refractive indices the approximation $n_{\text{min}} \leq n_{\text{approx}} \leq n_{\text{max}}$ which the theory predicts. In our calculations, we have that for $n_1$ we calculate  $n_{\text{approx}} = 3.7552$ and for  $n_2$ we calculate  $n_{\text{approx}} = 4.2231$
\begin{figure}[ht!]
\centering
\includegraphics[scale=0.18]{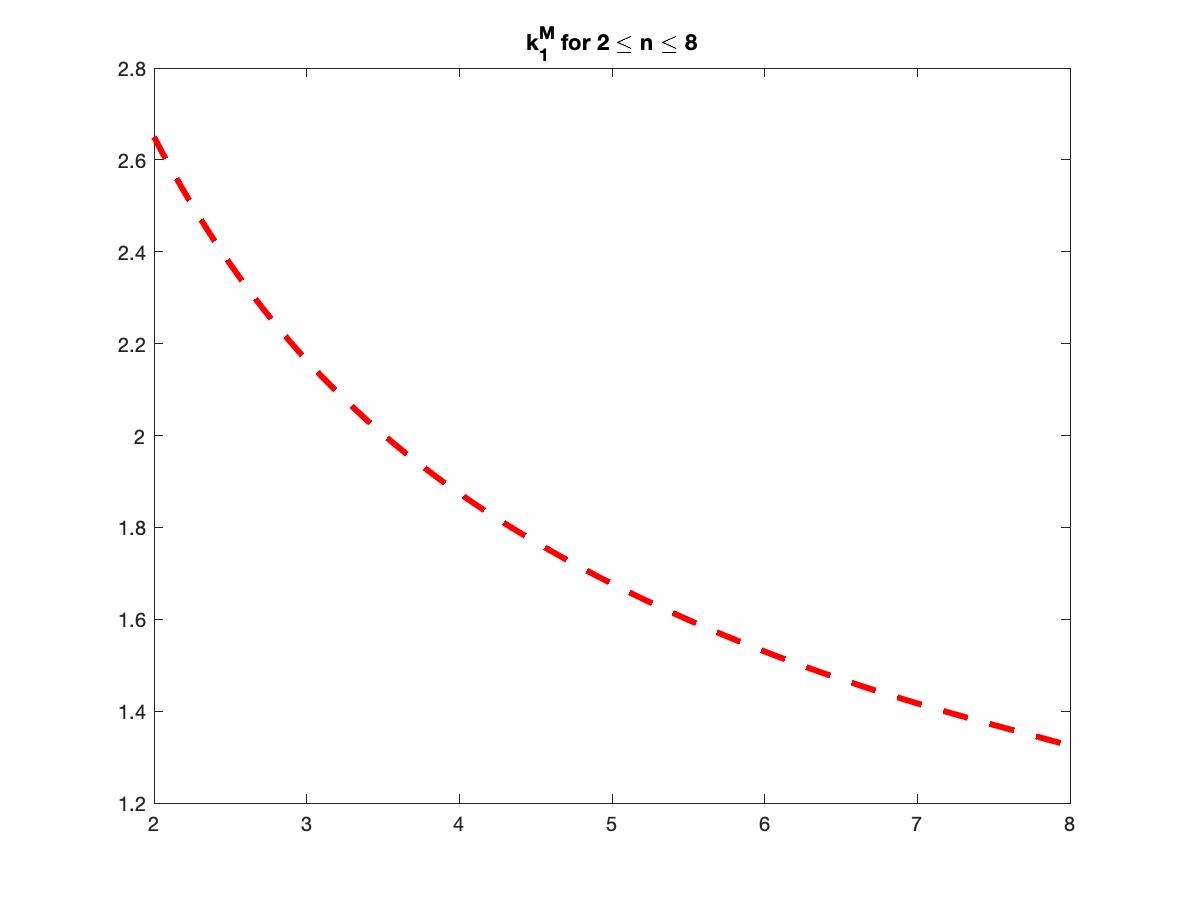}\includegraphics[scale=0.18]{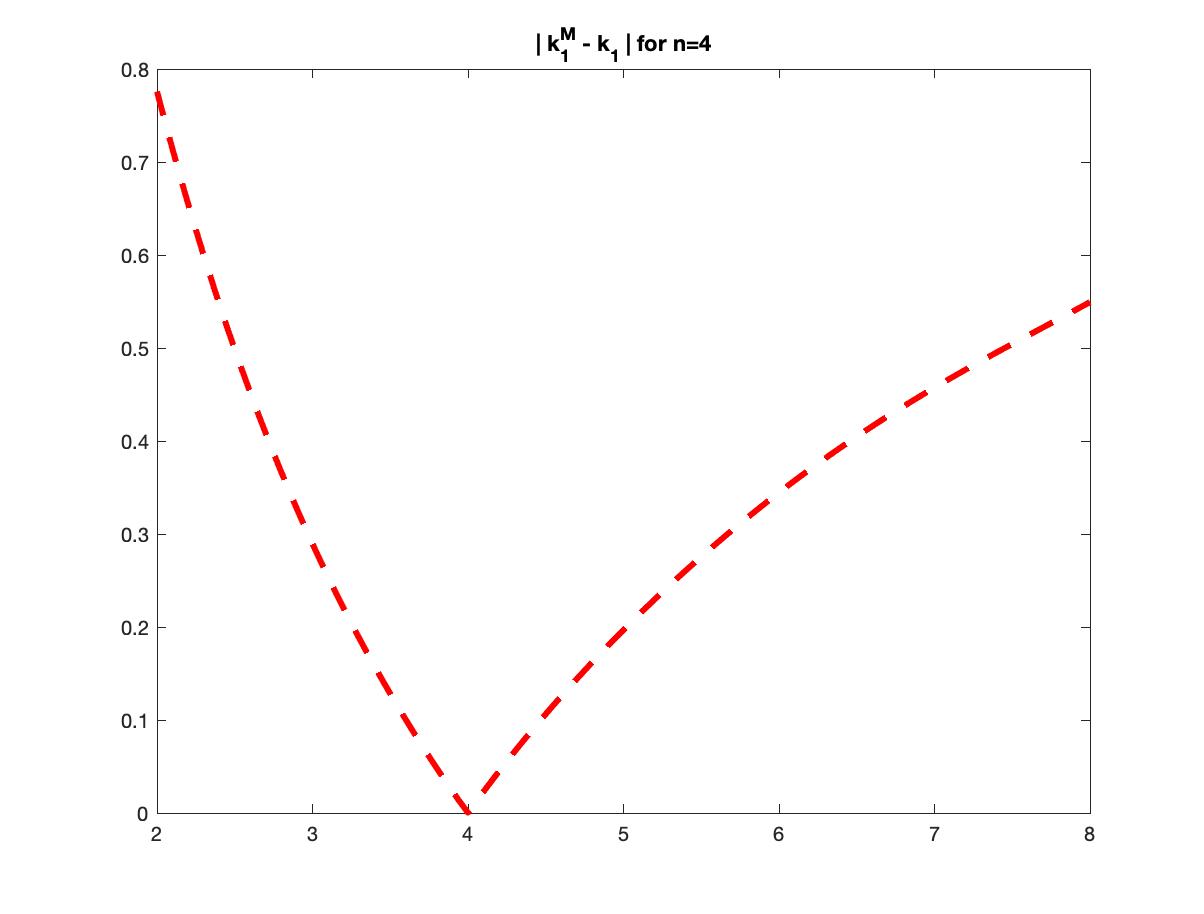}\\
\caption{Left: Plot of $n \mapsto k_1(n)$ for $n \in [2,8]$ using the Dirichlet-Spectral approximation. \\
Right: Plot of $n \mapsto \left| k_1(4) -  k_1(n)\right|$ where the reconstructed $n_{\text{approx}} = 3.9989$. } \label{monoplot}
 \end{figure}
\begin{figure}[ht!]
\centering
\includegraphics[scale=0.18]{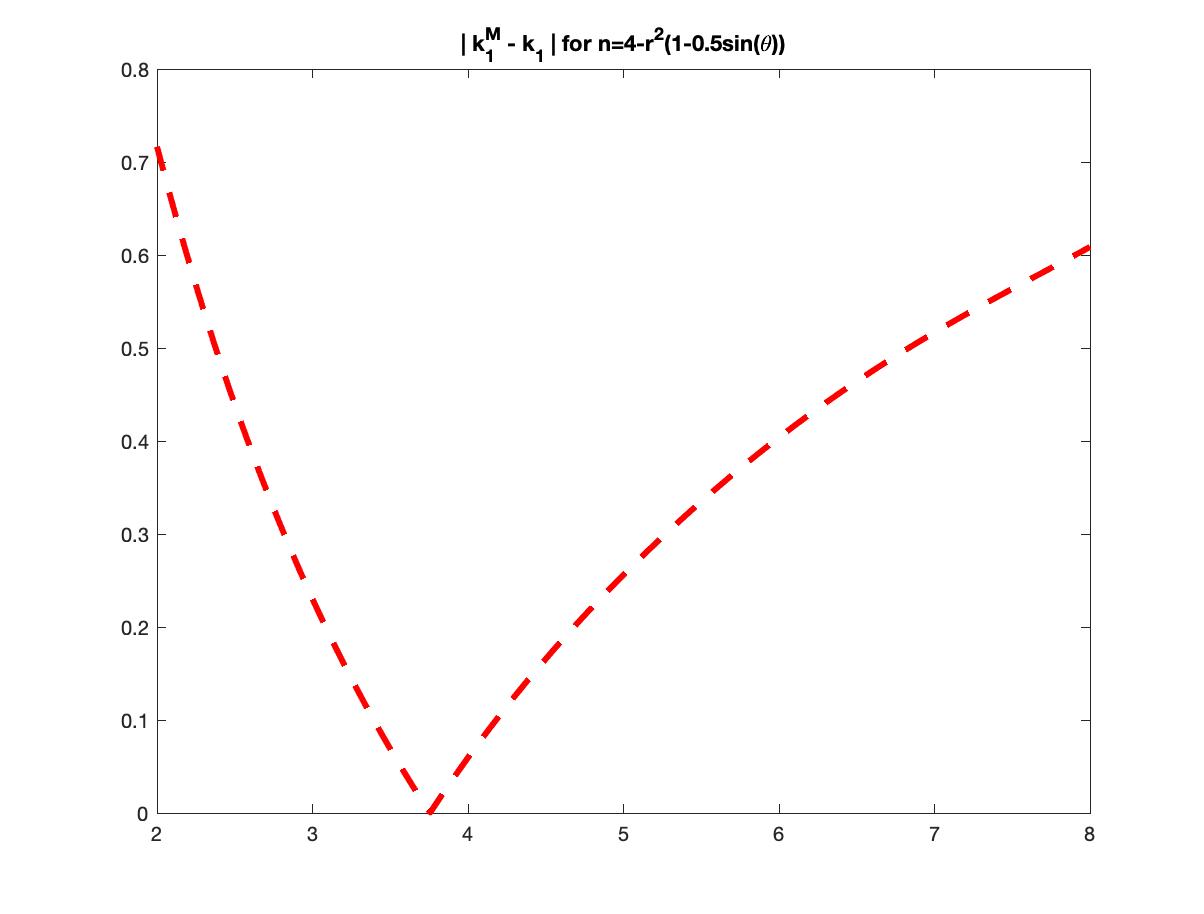}\includegraphics[scale=0.18]{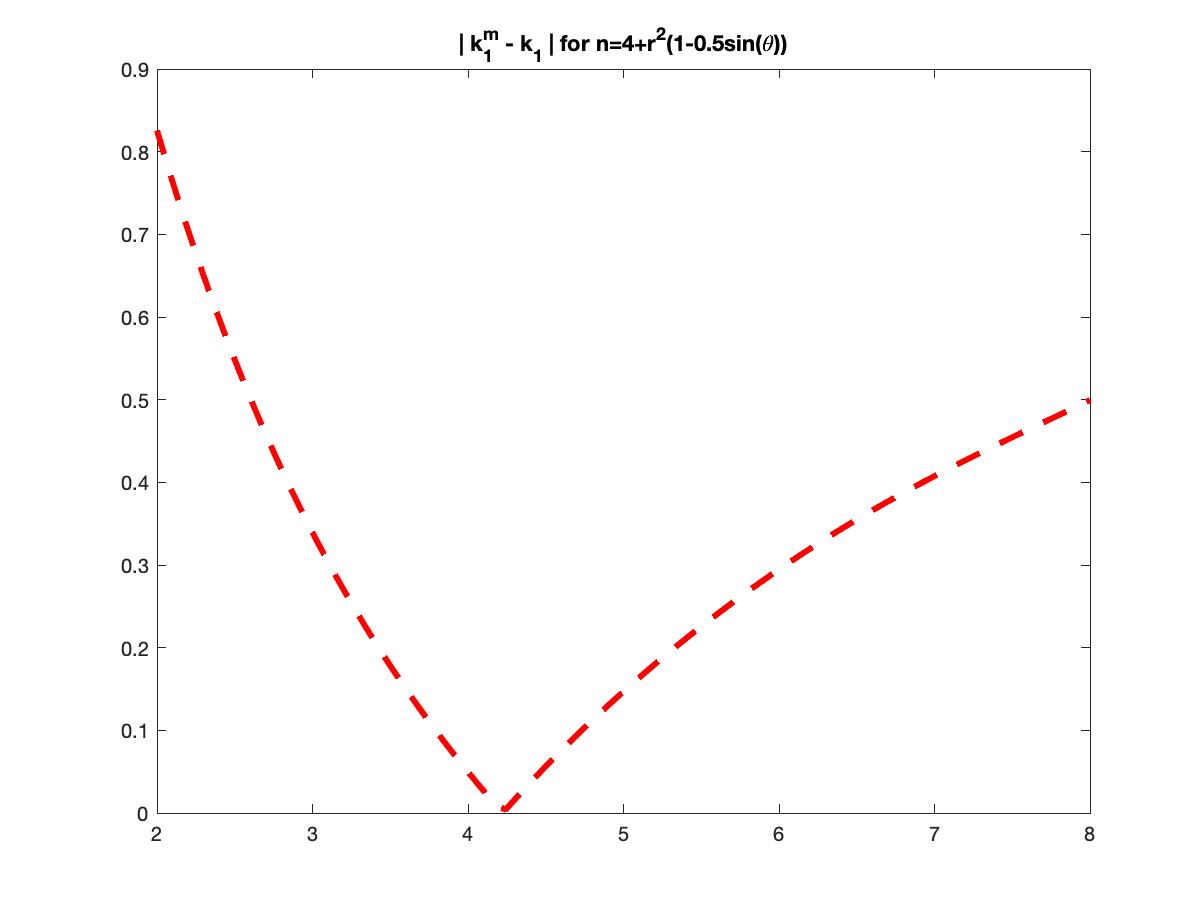}\\
\caption{Left: Plot of $n \mapsto \left| k_1(n_1) -  k_1(n)\right|$ for $n \in [2,8]$ where the reconstructed $n_{\text{approx}} = 3.7552$. \\
Right: Plot of $n \mapsto \left| k_1(n_2) -  k_1(n)\right|$ for $n \in [2,8]$ where the reconstructed $n_{\text{approx}} = 4.2231$. } \label{reconplot}
 \end{figure}

\section{Summary and Conclusions}
Here we have studied two interior transmission eigenvalue problems with the impedance boundary condition for the electromagnetic and acoustic scattering problems. For the inverse spectral problem we have proved monotonicity and limiting results with respect to the material parameters. Numerical examples are given to validate the theoretical monotonicity and convergence results for the acoustic  problem. We have developed a Dirichlet-Spectral approximation for the fourth order eigenvalue problem where more work is needed to study the accuracy of the method. One interesting question that arises from the analysis in this manuscript is what are the asymptotic expansion of the eigenvalues and eigenfunctions as $\eta$ tends to zero or infinity. Another interesting numerical and theoretical question that arises is can the Linear Sampling Method and/or  Inside-outside Duality Method recover the eigenvalues from the far-field pattern.


\end{document}